\documentclass[reqno]{amsart}

\usepackage{amsmath, amsfonts, mathtools, amsthm, amssymb}
\usepackage{amssymb,latexsym,amsmath,extarrows}
\usepackage{xcolor}
\usepackage{graphicx,mathrsfs,comment}
\usepackage{hyperref,url}
\usepackage{pict2e}
\usepackage{amstext}
\usepackage{bbm}
\usepackage{hyperref}
\hypersetup{
    colorlinks,
    linkcolor={black},
    citecolor={black},
    urlcolor={blue!80!black}
}
\usepackage{lipsum}

\newcommand{\erdos}{Erd{\H{o}}s}
\newcommand{\szemeredi}{Szemer{\'e}di}
\newcommand{\holder}{H\"older}

\let\epsilon\varepsilon
\numberwithin{equation}{section}
\newtheorem*{acknowledgement}{Acknowledgements}
\newtheorem{theorem}{Theorem}[section]
\newtheorem{lemma}[theorem]{Lemma}

\newtheorem{prop}[theorem]{Proposition}
\newtheorem{remark}[theorem]{Remark}

\newtheorem{notation}[theorem]{Notation}

\newtheorem{conjecture}[theorem]{Conjecture}

\begin{document}

\title{A Note on the Sum-Product Problem and the Convex Sumset Problem}

\author{Adam Cushman} \address{Adam Cushman\\ Indiana University Bloomington, USA} \email{acushma@iu.edu}

\begin{abstract}
We provide a new exponent for the Sum-Product conjecture on \(\mathbb{R} \). Namely for \(A \subset \mathbb{R} \) finite,
\[
    \max \left\{ \left\lvert A+A \right\rvert , \left\lvert AA \right\rvert  \right\} \gg_{\epsilon}   \left\lvert A \right\rvert ^{\frac{4}{3} + \frac{10}{4407} - \epsilon} 
.\]
We also provide new exponents for \(A \subset \mathbb{R} \) finite and convex, namely
\[
    \left\lvert A+A \right\rvert \gg_{\epsilon}  \left\lvert A \right\rvert ^{\frac{46}{29} - \epsilon},
\]
and
\[
    \left\lvert A-A \right\rvert \gg_{\epsilon}   \left\lvert A \right\rvert ^{\frac{8}{5} + \frac{1}{3440} -\epsilon}
.\]
\end{abstract}

\maketitle

\section{Introduction}

Let \(A,B \subset \mathbb{R} \). Their sum is defined as
\[
    A + B = \left\{ a + b : a \in A,~ b \in B \right\} 
.\]
We define similarly \(A - B\), \(A B\), and \(A / B\) (\(0 \not \in B\)).
The general question is, for finite \(A,B \subset \mathbb{R} \), how do the sizes of the sets
above depend on the sizes and structures of \(A,B\).

One of the most well-known open problems in this direction is the Sum-Product conjecture,
given by \erdos\ and \szemeredi\ in \cite{erdos-szemeredi}. It states that
regardless of the structure of \(A \subset \mathbb{R} \), either \(A+A\) or \(AA\) is large.

\begin{conjecture}[Sum-Product Conjecture]\label{conj:SP}
For all \(\epsilon > 0\), there exists \(c > 0\), such that for any finite set \(A \subset \mathbb{R} \),
\[
    \max \left( \left\lvert A+A \right\rvert ,\left\lvert AA \right\rvert  \right) \geq c \left\lvert A \right\rvert ^{2 - \epsilon}
.\]
\end{conjecture}

A related conjecture was given by \erdos\ in \cite{erdos1977}. It states that
if \(A \subset \mathbb{R} \) is convex, the sum and difference sets must be large.

\begin{conjecture}\label{conj:convex}
For all \(\epsilon > 0\), there exists \(c > 0\), such that for any finite convex set \(A \subset \mathbb{R} \),
\[
    \min \left( \left\lvert A+A \right\rvert , \left\lvert A-A \right\rvert  \right)  \geq c \left\lvert A \right\rvert ^{2 - \epsilon}
.\]
\end{conjecture}

A major breakthrough toward Conjecture \ref{conj:SP} was proving the case when the exponent \(2\)
is replaced by the exponent \(\frac{4}{3} \),
which was done by Solymosi in \cite{Solymosi}. A sequence of small improvements over \(\frac{4}{3} \) were
made by Konyagin-Shkredov in \cite{konyagin-shkredov}, Shakan in \cite{shakan}, Rudnev-Stevens in \cite{rudnev-stevens}, and
Bloom in \cite{bloom}. The current best exponent is due to Bloom, who obtained
\(\frac{4}{3} + \frac{2}{951} \).

In this paper, we provide another incremental improvement towards Conjecture \ref{conj:SP}.

\begin{theorem}\label{thm-SP}
For all \(\epsilon > 0\), there exists \(c > 0\) such that for any finite set \(A \subset \mathbb{R} \),
\[
    \max \left\{ \left\lvert A+A \right\rvert , \left\lvert AA \right\rvert  \right\} \geq c\left\lvert A \right\rvert ^{\frac{4}{3} + \frac{10}{4407} - \epsilon} 
.\]
\end{theorem}

We also provide improvements in the direction of Conjecture \ref{conj:convex}.
The best exponent for the sumset is due to Rudnev-Stevens in \cite{rudnev-stevens},
where they obtain the exponent \(\frac{30}{19} \). We provide an improvement to this.
\begin{theorem}\label{thm-CSUM}
For all \(\epsilon > 0\), there exists \(c > 0\) such that for any finite convex set \(A \subset \mathbb{R} \),
\[
    \left\lvert A+A \right\rvert \geq c\left\lvert A \right\rvert ^{\frac{46}{29} - \epsilon}
.\]
\end{theorem}

The best exponent for the difference set is stronger than that of the sumset.
Intuitively, this is because the difference set possesses more symmetry than the sumset.
Schoen-Shkredov in \cite{schoen-shkredov} proved the exponent \(\frac{8}{5} \) holds for the difference set, and this
was recently improved by Bloom in \cite{bloom} to \(\frac{8}{5} + \frac{1}{4175}\).
We provide another incremental improvement to this.
\begin{theorem}\label{thm-CDIFF}
For all \(\epsilon > 0\), there exists \(c > 0\) such that for any finite convex set \(A \subset \mathbb{R} \),
\[
    \left\lvert A-A \right\rvert \geq c \left\lvert A \right\rvert ^{\frac{8}{5} + \frac{1}{3440}  - \epsilon}
.\]
\end{theorem}

\bigskip

\subsection*{Outline of Proofs}

Improvements due to this paper are almost entirely contained in the following lemmas, which
are refinements of similar results appearing in previous literature. From the following lemmas,
we employ standard methods, given by \cite{rudnev-stevens} and \cite{bloom}, to obtain 
the results above.

Both lemmas involve projecting a set of ``rich'' elements in \(A\) to some ``popular''
elements in \(A-A\) or \(A+A\).
To state these lemmas we need the following standard definitions.
For any set \(S \subset \mathbb{R} \),
we call \(\delta_{S}(x) \) the representations of \(x\) as a difference
in \(S\), so
\[
    \delta_{S} (x) = \# \left\{ (s_1,s_2)\in S^{2} : x = s_1 - s_2 \right\} 
.\]
We call \(\sigma_{S} (x)\) the representations of \(x\) as a sum in \(S\), so 
\[
    \sigma_{S} (x) = \# \left\{ (s_1,s_2)\in S^{2} : x = s_1 + s_2 \right\} 
.\]
For \(k \in \mathbb{R} \), we define the \(k\)-th additive energy \(E_k(A)\) as
\[
    E_{k} (A) = \sum _{x \in A-A} \delta_{A} (x)^{k} 
.\]
We use \(\ll, \gg\) to suppress absolute constants, and \(\lesssim, \gtrsim \) to suppress
powers of \(\log \left\lvert A \right\rvert \), where \(A\) is clear. We use
\(X \approx Y\) to mean \(X \lesssim Y\) and \(Y \lesssim X\).

We proceed with the first lemma, corresponding to the set \(A-A\).
\begin{lemma}\label{lem:diff-proj}
For finite sets \(A \subset \mathbb{R} \), define the ``popular'' differences
\[
P =\left\{x \in A-A: \delta_A(x) \geq \frac{1}{11} \frac{|A|^2}{|A-A|}\right\}
.\]

We have
\begin{equation} \label{eq:diff-proj-eq}
    \left\lvert A \right\rvert ^{6} \ll E_3(A)\cdot \sum _{x \in P} \delta_{P}(x).
\end{equation}
\end{lemma}

Schoen-Shkredov in \cite{schoen-shkredov} obtain the exponent \(\frac{8}{5}\) for the convex difference set by proving
\begin{equation} \label{eq:schoen-shkredov-proj}
    \left\lvert A \right\rvert ^{6} \ll E_3(A) \cdot \sum _{x \in P} \delta_{(A-A)} (x),
\end{equation}
and using \szemeredi-Trotter bounds for the RHS. Lemma \ref{lem:diff-proj} is a
refinement of this, as \(A-A\) is replaced by a popular subset. Bloom in \cite{bloom}
obtained
\begin{equation} \label{eq:bloom-proj}
    \left\lvert A \right\rvert ^{6} \ll E_3(A) \cdot \sum _{x \in A-A} \delta _{P} (x)
\end{equation}
and provided a framework through which the improvement from (\ref{eq:schoen-shkredov-proj}) to
(\ref{eq:bloom-proj}) yields an improvement to Conjecture \ref{conj:convex}.
Using Lemma \ref{lem:diff-proj}, we follow Bloom's framework to obtain
Theorem \ref{thm-CDIFF}.

To prove Lemma \ref{lem:diff-proj}, we project triplets in \(A\) to pairs of differences in \(d_1,d_2 \in A-A\)
such that \(d_1- d_2 \in A-A\) using the following truism
\[
    (r - a_1) - (r - a_2) = a_2 - a_1
.\]
To obtain popular differences instead of ordinary differences, we use the idea of ``rich'' elements, provided by
Rudnev-Stevens in \cite{rudnev-stevens}. The rich elements are those which give a lot of popular differences, i.e.
\[
    R_{A} = \left\{ x \in A : \left\lvert (x - A)\cap P \right\rvert \geq \frac{2}{\sqrt{11} } \left\lvert A \right\rvert  \right\} 
.\]
It turns out that \(\left\lvert R_{A}  \right\rvert \gg \left\lvert A \right\rvert \). Moreover, we see that there are \(\gg \left\lvert A \right\rvert ^{2}\) pairs
\((a_1,a_2)\in A^{2}\) such that \(a_2 - a_1 \in P\), and for a fixed \(r \in R_{A} \) there are \(\gg \left\lvert A \right\rvert ^{2}\) pairs \((a_1,a_2) \in A^{2}\)
such that \(r-a_1,r-a_2 \in P\). We have chosen suitable constants in the definitions of \(P,R_{A} \)
so that by inclusion-exclusion, for any fixed \(r \in R_{A} \), there are \(\gg \left\lvert A \right\rvert ^{2}\) pairs \((a_1,a_2)\in A^{2}\)
such that
\[
    a_2-a_1 \in P,~ r-a_1 \in P,~ r-a_2 \in P
.\]
Seeing that \(\left\lvert R_{A}  \right\rvert \gg \left\lvert A \right\rvert \), we have a set of \(\gg \left\lvert A \right\rvert ^{3} \)
triples \((r,a_1,a_2)\) which map by
\[
    (r,a_1,a_2) \mapsto (r-a_1,r-a_2)
\]
to \(p_1,p_2 \in P\) such that \(p_1-p_2 \in P\). Applying Lemma \ref{lem:CS-proj} gives
Lemma \ref{lem:diff-proj}.

\medskip

We have a corresponding, slightly more technical, lemma for the sumset.
\begin{lemma}\label{lem:sum-proj}
For finite sets \(A,X \subset \mathbb{R} \) define the ``popular'' sums
\[
    P_{A} (X) = \left\{ y \in X + X : \sigma_{X} (y) \geq \frac{\left\lvert X \right\rvert ^{2}}{8  \left\lvert X+X \right\rvert \log \left\lvert A \right\rvert }  \right\} 
,\]
and the ``rich'' set
\[
    R_{A} (X) = \left\{ x \in X : \left\lvert \left( X + x \right) \cap P_{A} (X) \right\rvert \geq \frac{3}{4}  \left\lvert X \right\rvert  \right\} 
.\]
We have

(1)
For sufficiently large finite sets \(A \subset \mathbb{R} \), there exists \(B \subset A\) with \(\left\lvert B \right\rvert \geq \frac{1}{2} \left\lvert A \right\rvert \) such that
\[
    E_{\frac{12}{7} } (R_{A} (B)) \geq \frac{E_{\frac{12}{7} } (B)}{\log \left\lvert A \right\rvert } 
.\]

(2)
There is \(\Delta \in \mathbb{R} \) such that, defining
\[
    P_{\Delta} = \left\{ x : \delta_{R_{A} (B)} (x) \in [\Delta,2\Delta) \right\} 
,\]
we have
\[
    \Delta^{\frac{12}{7} } \left\lvert P_{\Delta}  \right\rvert \approx E_{\frac{12}{7} } (R_{A} (B))\approx E_{\frac{12}{7} } (B)
,\]
and moreover,
\begin{equation} \label{eq:sum-proj-eq}
    \Delta^{2} \left\lvert P_{\Delta}  \right\rvert ^{2} \left\lvert B \right\rvert ^{2} \ll E_3(B) \cdot \# \left\{ p_1 - p_2 = p_3 : p_1,p_2 \in P_{A} (B),~ p_3 \in P_{\Delta}  \right\}.
\end{equation}  
\end{lemma}

In the spirit of Lemma \ref{lem:diff-proj}, note that
\[
    \# \left\{ p_1 - p_2 = p_3 : p_1,p_2 \in P_{A} (B),~ p_3 \in P_{\Delta}  \right\}  = \sum _{x \in P_{\Delta} } \delta_{P_{A} (B)} (x)
.\]

Rudnev-Stevens in \cite{rudnev-stevens} prove a version of Lemma \ref{lem:sum-proj} where (\ref{eq:sum-proj-eq}) is replaced by
\[
    \Delta^{2} \left\lvert P_{\Delta}  \right\rvert ^{2} \left\lvert B \right\rvert ^{2} \ll E_3(B) \cdot \# \left\{ p_1 - s = p_2 : p_1 \in P_{A} (B),~ p_2 \in P_{\Delta},~ s \in B + B     \right\} 
.\]
Again, our improvement is effectively replacing \(B + B\) with a ``popular'' subset. We follow the 
framework of Rudnev-Stevens to obtain Theorem \ref{thm-CSUM} from Lemma \ref{lem:sum-proj}.

Rudnev-Stevens also provide a framework through which results of type Lemma \ref{lem:sum-proj} can
be turned into results on Conjecture \ref{conj:SP}. This framework was refined by Bloom in \cite{bloom},
leading to the best known Sum-Product exponent. We follow the work of Rudnev-Stevens
and Bloom to obtain Theorem \ref{thm-SP} from Lemma \ref{lem:sum-proj}.

Proving Lemma \ref{lem:sum-proj} is similar to proving Lemma \ref{lem:diff-proj}.
This proof is almost identical to that of \cite{rudnev-stevens}, the only difference being
in the use of inclusion-exclusion to gain an additional popular term.

For part \((1)\), we consider the sequence of sets defined by \(A_0 = A\), \(A_{j+1} = R_{A} (A_{j} )\),
and demonstrate using trivial bounds on \(E_{\frac{12}{7} } \)
that some \(A_{j} \), for \(j \leq \log \left\lvert A \right\rvert \), must be a suitable \(B\)
in the sense of (1). We obtain \(\Delta,P_{\Delta} \) by
dyadic pigeonholing on \(E_{\frac{12}{7} } (R_{A} (B))\).

For (2), we project triplets in \(B\) to pairs of sums \(s_1,s_2 \in B+B\) such that \(s_1-s_2 \in B-B\)
using the following truism
\[
    (r_1 + b) - (r_2 + b) = r_1 - r_2
.\]
There are \(\geq \Delta \left\lvert P_{\Delta}  \right\rvert \) many pairs \((r_1,r_2)\in R_{A} (B)\) such that
\(r_1 - r_2 \in P_{\Delta} \). For each \(r_{i} \), by the definition of \(R_{A} (B)\),
there are \(\geq \frac{3}{4} \cdot \left\lvert B \right\rvert \) choices of \(b\) such that \(r_{i}  + b \in P_{A} (B)\).
By inclusion-exclusion, there are \(\gg \left\lvert B \right\rvert \) choices of \(b \in B\) such that
\(r_1 + b,~ r_2 + b \in P_{A} (B)\). We then have \(\gg \Delta \left\lvert P_{\Delta}  \right\rvert \left\lvert B \right\rvert  \)  triples
\((r_1,r_2,b)\) which under the map
\[
    (r_1,r_2,b) \mapsto (r_1 + b, r_2 + b)
\] 
map to \(p_1,p_2 \in P_{A} (B)\) such that \(p_1 - p_2 \in P_{\Delta} \).
Using Lemma \ref{lem:CS-proj} gives the desired result.

We obtain sum-product results from these Lemmas by bounding the RHS of 
(\ref{eq:diff-proj-eq}) and (\ref{eq:sum-proj-eq}). 
In general, the strategy is first to use \holder's inequality to
bound these in terms of some additive energies.
These energies are then estimated by using the \szemeredi-Trotter theorem \cite{SzT-original}, an upper
bound for incidences between points and lines in the plane. It is in section \ref{section:SzT}
that we collect the energy bounds which are given by the \szemeredi-Trotter theorem.

\bigskip

\subsection*{Notation and Basic Results}
Let \(A\) be a finite set, and \(X(A),Y(A)\) be some quantities depending on \(A\), for example \(\left\lvert A\pm A \right\rvert \), or \(\left\lvert A \right\rvert \).

We say \(X(A) \ll Y(A)\) if \(X(A) = O(Y(A))\) as \(\left\lvert A \right\rvert \to \infty  \).
We say \(X(A) \asymp Y(A)\) if \(X(A) \ll Y(A)\) and \(Y(A) \ll X(A)\). We say \(X(A) \lesssim Y(A)\) if
there is \(c \in \mathbb{R} \) such that \(X(A) \ll Y(A) \log ^{c}\left\lvert A \right\rvert \), and we say
\(X(A) \approx Y(A)\) if \(X(A)\lesssim Y(A)\) and \(Y(A) \lesssim X(A)\).

For \(n \in \mathbb{N} \), we use the notation \([n] = \left\{ 1,2, \ldots , n \right\} \).

We say a finite set \(A = \left\{ a_1 < a_2 < \ldots < a_{n}  \right\} \subset \mathbb{R}  \) is convex
if the sequence \(\left\{ a_{j+1} - a_{j}  \right\} _{j=1} ^{n-1}\) is strictly increasing. 
For any finite \(A \subset \mathbb{R} \) convex, there is a strictly convex smooth function \(f\)
such that \(a_{j} = f(j) \) for all \(j \in [\left\lvert A \right\rvert ]\). 

Denote by \(r_{A+A} (x) \)
the number of representations of \(x\) in the form \(a_1 + a_2\), i.e.
\[
    r_{A+A} (x) = \# \left\{ (a_1,a_2)\in A^{2} : a_1 + a_2 = x \right\} 
.\]
We define \(r_{A-A} (x),~ r_{AA} (x),~ r_{\frac{A}{A} } (x)\) etc. all similarly.

Denote by \(\delta_{A,B}(x) = r_{A-B} (x) \), and \(\delta_{A} (x) = \delta_{A,A} (x)\). Similarly, denote by
\(\sigma_{A,B} (x) = r_{A + B} (x)\) and \(\sigma_{A} (x) = \sigma_{A,A} (x)\). Observe the trivial results
\[
    \left\lvert A \right\rvert \left\lvert B \right\rvert = \sum _{x} \delta_{A,B} (x) = \sum _{x}  \sigma_{A,B} (x)
.\]

We define
\[
    E(A,B) = \# \left\{ (a_1,a_2,b_1,b_2)\in A^{2}\times B^{2} : a_1-b_1 = a_2 - b_2 \right\}
\]
to be the additive energy. See that
\[
   E(A,B) = \sum _{x} \delta_{A,B} (x)^{2} = \sum _{x} \delta_{A} (x) \delta_{B} (x) = \sum _{x}  \sigma_{A,B} (x)^{2}
.\]

By Cauchy-Schwarz, we obtain the inequality
\[
    \left\lvert A \right\rvert \left\lvert B \right\rvert  = \sum _{x} r_{A \pm B} (x) \leq \left\lvert A \pm B \right\rvert ^{\frac{1}{2} } \left( \sum _{x} r_{A \pm B} (x)^{2} \right) ^{\frac{1}{2} }  = \left\lvert A \pm B \right\rvert ^{\frac{1}{2} } E(A,B) ^{\frac{1}{2} }
,\]
which relates the additive energy to the size of the sum and difference sets.

We generalize \(E(A,B)\) to higher energies by
\[
    E_{k} (A,B) : = \sum _{x}  \delta_{A,B} (x)^{k}
.\]
We call \(E_{k} (A) = E_{k} (A,A)\), and \(E(A) = E(A,A)\).

We also define the multiplicative energies as
\[
    E^{\times }_{k} (A,B) = \sum _{x} r_{\frac{A}{B} } (x)^{k}
.\]
We call \(E^{\times }(A,B)= E^{\times }_{2} (A,B)\). We call \(E_{k} ^{\times }(A) = E_{k} ^{\times }(A,A)\), and \(E^{\times }(A) = E^{\times }(A,A)\).

For real valued functions \(f\) with finite support, and for \(p \in [1,\infty )\), the \(\ell ^{p}\) norm of \(f\) is
defined by
\[
    \left\lVert f \right\rVert _{p} = \left( \sum _{x} \left\lvert f(x) \right\rvert ^{p} \right) ^{\frac{1}{p} }
.\]
For example, for \(k \in [1,\infty )\),
\[
    E_{k} (A,B)^{\frac{1}{k} } = \left\lVert \delta_{A,B}  \right\rVert _{k} 
.\]

Finally, we record the following ``projection'' lemma.
\begin{lemma}\label{lem:CS-proj}
Let \(f:X\to Y\) be a map between finite sets. We have
\[
    \left\lvert X \right\rvert ^{2} \leq \left\lvert Y \right\rvert \cdot \# \left\{ (x_1,x_2)\in X^{2} : f(x_1)= f(x_2) \right\}
.\]
\end{lemma}
\begin{proof}
Using Cauchy-Schwarz gives
\begin{align*}
    \left\lvert X \right\rvert & = \sum _{y \in Y} \# \left\{ x : f(x) = y \right\} \\
    & \leq \left\lvert Y \right\rvert^{\frac{1}{2} } \cdot \# \left\{ (x_1,x_2)\in X^{2} : f(x_1) = f(x_2) \right\}^{\frac{1}{2} } \qedhere
\end{align*}
\end{proof}

\medskip

\subsection*{Organization of Paper}
In section \ref{section:lemmas} we prove the key Lemmas \ref{lem:diff-proj} and \ref{lem:sum-proj}. In section 
\ref{section:SzT} we collect the standard convex \szemeredi-Trotter bounds, along
with the general bounds of \cite{rudnev-stevens} and \cite{bloom}. Sections \ref{section:SP}, \ref{section:CSUM}, and
\ref{section:CDIFF} follow the methods of Rudnev-Stevens and Bloom to obtain
Theorems \ref{thm-SP}, \ref{thm-CSUM}, and \ref{thm-CDIFF} respectively from the lemmas obtained in section \ref{section:lemmas}.

\begin{acknowledgement}
I am extremely grateful to Shukun Wu for the many conversations and sustained support he has provided,
and also for getting me interested in this problem. I thank Ilya Shkredov for his
time and his helpful guidance. I thank Thomas Bloom for answering my questions about his work.
Finally, I thank the Department of Mathematics at Indiana University and its faculty for hosting the REU program
during which I began working on this project.
\end{acknowledgement}

\bigskip

\section{Proof of Key Lemmas}\label{section:lemmas}
Recall that we have defined
\[
P =\left\{x \in A-A: \delta_A(x) \geq \frac{1}{11} \frac{|A|^2}{|A-A|}\right\}
.\]
By the definition of \(P\),
\[
    \sum _{x \not \in P} \delta_{A} (x) < \frac{1}{11} \frac{\left\lvert A \right\rvert ^{2}}{\left\lvert A-A \right\rvert } \cdot \left\lvert A-A \right\rvert 
,\]
and hence
\[
\sum_{x \in P } \delta_A(x) \geq \frac{10}{11}|A|^2
.\]

Define the ``rich'' elements of \(A\) to be
\[
R_{A}  =\left\{x \in A:|(x-A) \cap P | \geq \frac{2}{\sqrt{11}}|A|\right\}
.\]
See that with these definitions of \(R_{A} , P \), we have \(\left\lvert R_{A}  \right\rvert \gg \left\lvert A \right\rvert  \). 

Indeed, we partition \( \sum _{x \in P} \delta_{A} (x)\) as
\begin{align*}
    \frac{10}{11} \left\lvert A \right\rvert ^{2} & \leq \sum _{x \in P} \delta_{A} (x) = \# \left\{ (a_1,a_2)\in A^{2}: a_1-a_2 \in P \right\}\\
    & = \# \left\{ (r,a)\in R_{A} \times A  : r - a \in P\right\} + \# \left\{ (n,a)\in \left( A \setminus R_{A}  \right) \times A : n - a \in P \right\} 
\end{align*}
Bounding the first term trivially, and the second term using the definition of \(R_{A} \), we obtain
\begin{align*}
    \frac{10}{11} \left\lvert A \right\rvert ^{2} & \leq \left\lvert R_{A}  \right\rvert \left\lvert A \right\rvert  + \frac{2}{\sqrt{11} } \left\lvert A \right\rvert \left( \left\lvert A \right\rvert - \left\lvert R_{A}  \right\rvert  \right) \\
    &  = \left\lvert R_{A}  \right\rvert \left\lvert A \right\rvert \left( 1 - \frac{2}{\sqrt{11} }  \right) + \frac{2}{\sqrt{11} } \left\lvert A \right\rvert ^{2}.
\end{align*}
Simplifying gives
\[
|R_{A} |\geq \left(\frac{10}{11}-\frac{2}{\sqrt{11}}\right) \cdot \frac{1}{1-\frac{2}{\sqrt{11}}}|A| > \frac{1}{2}|A|
.\]

Define
\[
S_1=\left\{\left(a_1, a_2\right) \in A^2: a_1-a_2 \in P \right\}
,\]
and
\[
S_r=\left\{\left(a_1, a_2\right) \in A^2: r-a_1, r-a_2 \in P \right\}
.\]
By the definition of \(P\),
\[
\left|S_1\right|=\sum_{x \in P} \delta_A(x) \geq \frac{10}{11}|A|^2
.\]
By the definition of \(R_{A} \),
\[
\left|S_r\right| \geq\left(\frac{2}{\sqrt{11}}|A|\right)^2 \geq \frac{4}{11}|A|^2
.\]

Therefore, for any \(r \in R_{A} \),
\[
\begin{aligned}
\left|S_1 \cap S_r\right| & =\left|S_1\right|+\left|S_r\right|-\left|S_1 \cup S_r\right| \\
& \geq \frac{10}{11}|A|^2+\frac{4}{11}|A|^2-|A|^2 \\
& \gg|A|^2 .
\end{aligned}
\]

We have shown that for any \(r \in R_{A} \),
\[
\left|\left\{\left(a_1, a_2\right) \in A^2: r-a_1, r-a_2, a_1-a_2 \in P \right\}\right| \gg|A|^2
,\]
or equivalently
\[
\left|\left\{\left(r, a_1, a_2\right) \in R_{A}  \times A^2: r-a_1, r-a_2, a_1-a_2 \in P \right\}\right| \gg|R_{A} ||A|^2 \gg|A|^3
.\]

We now project these triplets \((r,a_1,a_2)\) to pairs \((p_1,p_2)\in P^{2}\) for which \(p_1-p_2 \in P\).
Let
\[
S=\left\{\left(r, a, a^{\prime}\right) \in R_{A}  \times A^2: r-a, r-a^{\prime}, a-a^{\prime} \in P  \right\}
.\]
We have just demonstrated that \(\left\lvert S \right\rvert \gg \left\lvert A \right\rvert ^{3} \). Define the map \(f : A ^{3} \to (A-A) ^{2}\) by
\[
f:\left(r, a, a^{\prime}\right) \mapsto\left(r-a^{\prime}, r-a\right)
.\]
By the definition of \(S\), \(f(S) \subset Y\), where
\[
Y=\left\{\left(p_1, p_2\right) \in P ^2: p_1-p_2 \in P \right\}
.\]
We wish to apply Lemma \ref{lem:CS-proj}, and so we must consider
\[
    \left\lvert \left\{ (s_1,s_2)\in S^{2} : f(s_1)= f(s_2) \right\}  \right\rvert 
.\]
If we let \(s_i=\left(r_i, a_i, a_i^{\prime}\right)\), simply using the definition of \(f\) gives
\begin{align*}
    f(s_1) = f(s_2)&  \iff r_1-a_1=r_2-a_2,~ r_1-a_1^{\prime}=r_2-a_2^{\prime}\\
    & \iff r_1-r_2=a_1-a_2=a_1^{\prime}-a_2^{\prime}
\end{align*}
Therefore, using the fact that \(S \subset A^{3}\),
\begin{align*}
\# \left\{ (s_1,s_2) \in S^{2}  : f(s_1) = f(s_2)\right\}  = \# \left\{ (s_1,s_2)\in S^{2} : r_1-r_2 = a_1-a_2 = a_1' - a_2' \right\}      \\
 \leq \# \left\{\left(x_1, \ldots, x_6\right) \in A^6: x_1-x_2=x_3-x_4=x_5-x_6\right\} =E_3(A).
\end{align*}

We will now apply Lemma \ref{lem:CS-proj}. We get
\[
|S|^2 \leq|Y|\left|\left\{\left(s_1, s_2\right) \in S^2: f\left(s_1\right)=f\left(s_2\right)\right\}\right|
.\]
Using \(\left\lvert S \right\rvert \gg \left\lvert A \right\rvert ^{3} \) and
\[
    \left|\left\{\left(s_1, s_2\right) \in S^2: f\left(s_1\right)=f\left(s_2\right)\right\}\right| \leq E_3(A)
\]
gives
\[
\left\lvert A \right\rvert ^{6} \ll E_3(A) \left\lvert Y \right\rvert
.\]
Seeing that, by definition,
\[
    \left\lvert Y \right\rvert  = \sum _{x \in P} \delta_{P} (x)
,\]
Lemma \ref{lem:diff-proj} is proven. We proceed with Lemma \ref{lem:sum-proj},
which is argued similarly.

Let \(A\) be a sufficiently large (to be defined later) finite set.
We first demonstrate the existence of \(B \subset A\) with \(\left\lvert B \right\rvert \geq \frac{1}{2} \left\lvert A \right\rvert \) and
\[
    E_{\frac{12}{7} } \left( R_{A} (B) \right) \geq \frac{E_{\frac{12}{7} } (B)}{\log \left\lvert A \right\rvert }
.\]

Recall that we defined
\[
P_{A} (X) = \left\{ y \in X + X : \sigma_{X} (y) \geq \frac{\left\lvert X \right\rvert ^{2}}{8  \left\lvert X+X \right\rvert \log \left\lvert A \right\rvert }  \right\} 
\]
and
\[
    R_{A} (X) = \left\{ x \in X : \left\lvert \left( X + x \right) \cap P_{A} (X) \right\rvert \geq \frac{3}{4}  \left\lvert X \right\rvert  \right\} 
.\]
Observe that, by the definition of \(P_{A} (X)\),
\[
    \sum _{y \not \in P_{A} (X)} \sigma_{X} (y) < \frac{\left\lvert X \right\rvert ^{2}}{8 \log \left\lvert A \right\rvert } 
,\]
so
\[
    \sum _{y \in P_{A} (X)} \sigma_{X} (y) \geq \left\lvert X \right\rvert ^{2} \left( 1 - \frac{1}{8 \log \left\lvert A \right\rvert }  \right)
.\]

We show that \(R_{A} (X)\) is ``large''. We have just shown
\begin{equation} \label{eq:sum-proj-rich-bound}
    \left\lvert X \right\rvert ^{2} \left( 1 - \frac{1}{8 \log \left\lvert A \right\rvert }  \right)  \leq \sum _{y \in P_{A} (X)} \sigma_{X} (y) = \# \left\{ (x_1,x_2) \in X^{2}: x_1+x_2 \in P_{A} (X) \right\}.
\end{equation}
We partition the set in the RHS into
\[
    \left\{ (r,x)\in R_{A} (X) \times X : r+x \in P_{A} (X)  \right\} \bigsqcup  \left\{ (n,x) \in \left( X \setminus R_{A} (X) \right) \times X : n + x \in P_{A} (X)\right\} 
.\]
Using the trivial bound on the first set, and the definition of \(R_{A} (X)\) for the second set,
substituting into (\ref{eq:sum-proj-rich-bound}) gives
\[
   \left\lvert X \right\rvert ^{2} \left( 1 - \frac{1}{8 \log \left\lvert A \right\rvert }  \right)  \leq \left\lvert R_{A} (X) \right\rvert \left\lvert X \right\rvert + \left( \left\lvert X \right\rvert - \left\lvert R_{A} (X) \right\rvert  \right) \cdot \frac{3}{4} \cdot  \left\lvert X \right\rvert,
\]
which by simplifying gives
\[
    \left\lvert R_{A} (X) \right\rvert \geq \left( 1 - \frac{1}{2 \log \left\lvert A \right\rvert }  \right) \left\lvert X \right\rvert 
.\]

Suppose now, for the sake of contradiction, that for all \(B \subset A\) with \(\left\lvert B \right\rvert \geq \frac{1}{2} \left\lvert A \right\rvert \),
\[
    E_{\frac{12}{7} } \left( R_{A} (B) \right) < \frac{E_{\frac{12}{7} } (B)}{\log \left\lvert A \right\rvert } 
.\]
We apply \(R_{A} \) iteratively. Namely, consider the sequence of sets \(\left\{ A_{i}  \right\} \) defined by \(A_{0} = A\) and \(A_{i+1} = R_{A} (A_{i} )\).
See that for \(i \leq \log \left\lvert A \right\rvert \),
\[
    \left\lvert A_{i} \right\rvert  = \left\lvert R_{A} ^{(i)}(A) \right\rvert \geq \left( 1 - \frac{1}{2 \log \left\lvert A \right\rvert }  \right) ^{i} \left\lvert A \right\rvert \geq \left( 1 - \frac{1}{2 \log \left\lvert A \right\rvert }  \right) ^{\log \left\lvert A \right\rvert } \cdot \left\lvert A \right\rvert \geq \frac{1}{2} \cdot \left\lvert A \right\rvert 
,\]
where the last inequality holds for \(A\) sufficiently large. Note that this is the
first of 2 thresholds on the size of \(A\).

Therefore, for all \(i \leq \log \left\lvert A \right\rvert \), because \(\left\lvert A_{i}  \right\rvert \geq \frac{1}{2} \cdot \left\lvert A \right\rvert  \),
we have supposed for contradiction that
\[
    E_{\frac{12}{7} } \left( A_{i+1}  \right) = E_{\frac{12}{7} } \left( R_{A} (A_{i} ) \right) < \frac{E_{\frac{12}{7} } (A_{i} )}{\log \left\lvert A \right\rvert } 
,\]
which upon iterating gives
\begin{equation} \label{eq:sum-proj-contradiction}
    E_{\frac{12}{7} } \left( A_{\left\lfloor \log \left\lvert A \right\rvert  \right\rfloor }  \right) < \frac{E_{\frac{12}{7} } (A)}{\left( \log \left\lvert A \right\rvert \right)  ^{\left\lfloor\log  \left\lvert A \right\rvert  \right\rfloor }}.
\end{equation}  
Trivially, we have, for any set \(Z \subset \mathbb{R} \), \(\left\lvert Z \right\rvert ^{2} \leq E_{\frac{12}{7} } (Z) \leq \left\lvert Z \right\rvert ^{3} \).
Indeed
\[
    \left\lvert Z \right\rvert ^{2} = \sum _{x} \delta_{Z} (x) \leq \sum _{x} \delta_{Z} (x)^{\frac{12}{7} } = E_{\frac{12}{7} } \left( Z \right)
,\]
and
\[
    E_{\frac{12}{7} } (Z) = \sum _{x} \delta_{Z} (x)^{\frac{12}{7} } \leq \left\lvert Z \right\rvert ^{\frac{5}{7} } \sum _{x} \delta_{Z} (x) \leq   \left\lvert Z \right\rvert ^{3}
.\]
Using these bounds in (\ref{eq:sum-proj-contradiction}) gives
\[
    \frac{1}{4} \left\lvert A \right\rvert ^{2} \leq \left\lvert A_{\left\lfloor \log \left\lvert A \right\rvert  \right\rfloor }  \right\rvert ^{2} \leq E_{\frac{12}{7} } \left( A_{\left\lfloor \log \left\lvert A \right\rvert  \right\rfloor }  \right) <\frac{ E_{\frac{12}{7} }(A)}{\left( \log \left\lvert A \right\rvert \right)  ^{\left\lfloor \log \left\lvert A \right\rvert  \right\rfloor }}   < \frac{\left\lvert A \right\rvert ^{3} }{\left( \log \left\lvert A \right\rvert \right)  ^{\left\lfloor \log \left\lvert A \right\rvert  \right\rfloor }} 
,\]
a contradiction for \(A\) sufficiently large. This is the second and final threshold on the size of \(A\).

We have demonstrated that for finite \(A\) larger than some absolute constant, there is \(B \subset A\) with \(\left\lvert B \right\rvert \geq \frac{1}{2} \left\lvert A \right\rvert \) and
\[
    E_{\frac{12}{7} } (R_{A} (B)) \geq \frac{E_{\frac{12}{7} } (B)}{\log \left\lvert A \right\rvert } 
.\]

A standard dyadic pigeonholing argument on \(E_{\frac{12}{7} } (R_{A} (B))\) gives the existence of \(\Delta \in \mathbb{R} ^{+}\) such that,
defining
\[
    P_{\Delta} = \left\{ x : \delta_{R_{A} (B)}(x) \in [\Delta,2\Delta) \right\} 
,\]
we have
\[
    \Delta^{\frac{12}{7} } \left\lvert P_{\Delta}  \right\rvert \approx E_{\frac{12}{7} } (R_{A} (B)) \approx E_{\frac{12}{7} } (B)
.\]
It remains to be shown that
\[
    \Delta^{2} \left\lvert P_{\Delta}  \right\rvert ^{2}\left\lvert B \right\rvert ^{2} \ll E_{3} (B) \cdot \# \left\{ p_1 - p_2 = p_3 : p_1,p_2 \in P_{A} (B),~ p_3 \in P_{\Delta}  \right\} 
.\]

This follows by an almost identical projection as in the proof of Lemma \ref{lem:diff-proj}.
Define
\[
    X = \left\{ (r_1,r_2,b)\in R_{A} (B)^{2}\times B : r_1 + b \in P_{A} (B) ,~ r_2 + b \in P_{A} (B),~ r_1 -r_2 \in P_{\Delta}  \right\} 
.\]
Define \(f : B^{3} \to \left( B+ B \right) ^{2}\) by
\[
    f : (r_1,r_2,b) \mapsto (r_1 + b , r_2 + b)
.\]
By the definition of \(X\), \(f(X) \subset Y\) where
\[
    Y = \left\{ (p_1,p_2)\in P_{A} (B)^{2} : p_1 - p_2 \in P_{\Delta}  \right\} 
.\]
Lemma \ref{lem:CS-proj} gives
\[
    \left\lvert X \right\rvert ^{2} \leq \left\lvert Y \right\rvert \# \left\{ (x_1,x_2)\in X^{2} : f(x_1) = f(x_2) \right\} 
.\] 
Notice that, by the exact same argument as before,
\[
    \# \left\{ \left( x_1,x_2 \right) \in X^{2} : f(x_1) = f(x_2) \right\} \leq E_3(B)
,\]
so
\[
    \left\lvert X \right\rvert ^{2}\leq E_3(B) \left\lvert Y \right\rvert 
.\]
By definition,
\[
    \left\lvert Y \right\rvert = \sum _{x \in P_{\Delta} } \delta_{P_{A} (B)} (x)=\# \left\{ p_1 - p_2 = p_3 : p_1,p_2 \in P_{A} (B), p_3 \in P_{\Delta} \right\} 
,\]
so it remains to show that \(\left\lvert X \right\rvert \gg \Delta \left\lvert P_{\Delta}  \right\rvert \left\lvert B \right\rvert \).

Notice that, by the definition of \(\Delta,P_{\Delta} \),
\[
    \Delta \left\lvert P_{\Delta}  \right\rvert \asymp \# \left\{ (r_1,r_2)\in R_{A} (B)^{2} : r_1 - r_2 \in P_{\Delta}  \right\} 
.\]
Call
\[
    X_{r} = \left\{ x \in B : r + x \in P_{A} (B) \right\} 
,\]
and see that for any \(r \in R_{A} (B)\), by the definition of \(R_{A} (B)\), \(\left\lvert X_{r}  \right\rvert \geq \frac{3}{4}  \left\lvert B \right\rvert    \).
Using Inclusion-Exclusion, for any \(r_1,r_2 \in R_{A} (B)\),
\[
    \left\lvert X_{r_1} \cap X_{r_2}  \right\rvert \gg \left\lvert B \right\rvert
.\]
See that we may partition \(X\) as
\[
    X = \bigsqcup _{\substack{ (r_1,r_2) \in R_{A} (B)^{2} \\ r_1 - r_2 \in P_{\Delta}  }} \left\{ (r_1,r_2,b) : b \in X_{r_1} \cap X_{r_2}  \right\} 
,\]
and hence
\[
    \left\lvert X \right\rvert = \sum _{\substack{ (r_1,r_2)\in R_{A} (B)^{2} \\ r_1 - r_2 \in P_{\Delta}  }} \left\lvert X_{r_1} \cap X_{r_2}  \right\rvert \gg \left\lvert B \right\rvert \Delta \left\lvert P_{\Delta}  \right\rvert 
.\]
We have demonstrated that
\[
    \Delta^{2} \left\lvert P_{\Delta}  \right\rvert ^{2} \left\lvert B \right\rvert ^{2} \ll E_3(B) \cdot \# \left\{ p_1 - p_2 = p_3 : p_1,p_2 \in P_{A} (B),~ p_3 \in P_{\Delta}  \right\} 
,\]
so Lemma \ref{lem:sum-proj} is proven.

\bigskip

\section{\szemeredi-Trotter Lemmas}\label{section:SzT}
This section develops general energy estimates using the \szemeredi-Trotter theorem.
This section contains only restatements of existing results.
The following lemma is a slight modification of the \szemeredi-Trotter theorem, a proof
of which can be found in \cite{solymosi-szt}. Under the additional assumption that the point set
\(P\) is a Cartesian product, and the line set \(L\) contains no lines parallel to the axes,
one can omit the \(+ \left\lvert P \right\rvert \) term in the \szemeredi-Trotter theorem.

\begin{lemma}\label{lem:solymosi-szt}    
Let \(A,B \subset \mathbb{R}\) be finite, let \(P = A \times B\), and let \(L\) be either

(a) A finite set of lines whose slopes are finite nonzero real numbers

(b) Finitely many translates of a smooth convex curve

Then the number of incidences between \(P\) and \(L\) is \(\ll \left\lvert P \right\rvert ^{\frac{2}{3} }\left\lvert L \right\rvert ^{\frac{2}{3} } + \left\lvert L \right\rvert     \)
\end{lemma}

This yields 2 energy bounds which will be important, one in the convex case and one in the general case.
We provide the convex result first.

\begin{lemma}\label{lem:convex-szt}
Let \(A \subset \mathbb{R} \) be convex. We have

(1) For all \(B \subset \mathbb{R} \),
\[
    E_3(A,B)\lesssim \left\lvert A \right\rvert \left\lvert B \right\rvert ^{2}
.\]

(2) For all \(s \in (1,3)\),
\[
    E_{s} (A,B) \lesssim \left\lvert A \right\rvert \left\lvert B \right\rvert ^{\frac{s+ 1}{2} }
.\]
\end{lemma}

\begin{proof}
Fix \(B \subset \mathbb{R} \). We will prove (1) first, and (2) will follow immediately
after by interpolating for \(E_{s}(A,B) \) between \(E_3(A,B)\) and \(E_1(A,B)= \left\lvert A \right\rvert \left\lvert B \right\rvert  \). 

A standard dyadic partitioning gives \(k \in \mathbb{N} \) and
\[
    D_{k} = \left\{ x \in  A- B : \delta_{A,B} (x) \in [k,2k) \right\}  
\]
such that
\[
    k^{3} \left\lvert D_{k}  \right\rvert \approx E_3(A,B)
.\]

By the definition of \(D_{k} \),
\[
    k \left\lvert D_{k}  \right\rvert \leq \sum _{x \in D_{k} } \delta_{A,B} (x)
.\]
Since \(A\) is convex, let \(f\) be the convex function for which \(A = \left\{ f(j) : j \in \left[ \left\lvert A \right\rvert  \right]  \right\} \).
We have
\[
    \sum _{x \in D_{k} } \delta_{A,B} (x) = \sum _{x \in A} \sigma_{D_{k} ,B} (x) = \sum _{ n \in [\left\lvert A \right\rvert ]} \sigma_{D_{k} ,B} \left( f(n) \right) 
.\]

Fix some \(n \in [\left\lvert A \right\rvert ]\). If \(n \leq \frac{1}{2} \left\lvert A \right\rvert \), there are \(\gg \left\lvert A \right\rvert   \) solutions to
\begin{equation} \label{eq:n-difference-split}
    n = m_1 - m_2 : m_1,m_2 \in [\left\lvert A \right\rvert ].
\end{equation}
If \(n\geq  \frac{1}{2} \cdot \left\lvert A \right\rvert \), there are \(\gg \left\lvert A \right\rvert  \)
solutions to
\begin{equation} \label{eq:n-sum-split}
    n = m_1 + m_2 : m_1,m_2 \in [\left\lvert A \right\rvert ].
\end{equation}
It is plain that at least one of
\begin{equation} \label{eq:convex-wlog}
    \sum _{n \leq \frac{1}{2} \cdot \left\lvert A \right\rvert } \sigma_{D_{k} ,B} (f(n)) \geq \frac{1}{2} \cdot  \sum _{n \in [\left\lvert A \right\rvert ]} \sigma_{D_{k},B }(f(n)) 
\end{equation}
or
\[
    \sum _{n \geq  \frac{1}{2} \cdot \left\lvert A \right\rvert } \sigma_{D_{k} ,B} (f(n))\geq \frac{1}{2} \cdot  \sum _{n \in [\left\lvert A \right\rvert ]} \sigma_{D_{k},B }(f(n)) 
\]
must hold.
We assume (\ref{eq:convex-wlog}). If it is the latter, we could argue the following claim
in exactly the same way, using (\ref{eq:n-sum-split}) as opposed to (\ref{eq:n-difference-split}).

Since, for each \(n \leq \frac{1}{2} \cdot \left\lvert A \right\rvert \) there are \(\gg \left\lvert A \right\rvert \) representation
of \(n\) as \(m_1-m_2\), we have
\[
    \sum _{m_1,m_2 \in [\left\lvert A \right\rvert ]} \sigma_{D_{k} ,B} (f(m_1-m_2)) \gg \left\lvert A \right\rvert  \sum _{n \leq \frac{1}{2} \cdot \left\lvert A \right\rvert } \sigma_{D_{k} ,B} (f(n))
.\]
Substituting using (\ref{eq:convex-wlog}),
\[
    \sum _{m_1,m_2 \in [\left\lvert A \right\rvert ]} \sigma_{D_{k} ,B} (f(m_1-m_2)) \gg \left\lvert A \right\rvert \sum _{n \in [\left\lvert A \right\rvert ]} \sigma_{D_{k} ,B} (f(n))   
.\]

We observe that
\[
    \sum _{ m_1,m_2 \in [\left\lvert A \right\rvert ]} \sigma_{D_{k} ,B} \left( f(m_1-m_2) \right)
\]
counts solutions to
\begin{equation} \label{eq:convex-line-translations}
    f(m_1 - m_2) - d = b : m_1,m_2 \in [\left\lvert A \right\rvert ],~ d \in D_{k} ,~ b \in B.
\end{equation}
Solutions to (\ref{eq:convex-line-translations}) are precisely incidences of the point set \([\left\lvert A \right\rvert ]\times B \) with the set
of curves given by
\[
    \ell (x) = f(x - n) - d ,~ n \in [\left\lvert A \right\rvert ] ,~ d \in D_{k} 
,\]
which are translations of the curve given by \(f\). It follows by Lemma \ref{lem:solymosi-szt}
that
\[
    \sum _{ m_1,m_2 \in [\left\lvert A \right\rvert ]} \sigma_{D_{k} ,B} \left( f(m_1-m_2) \right) \ll \left( \left\lvert A \right\rvert ^{2} \left\lvert D_{k}  \right\rvert \left\lvert B \right\rvert  \right) ^{\frac{2}{3} } + \left\lvert A \right\rvert \left\lvert D_{k}  \right\rvert 
.\]

The trivial bound \(\left\lvert D_{k}  \right\rvert \leq \left\lvert A \right\rvert \left\lvert B \right\rvert \) gives
\[
    \left\lvert A \right\rvert \left\lvert D_{k}  \right\rvert \ll \left( \left\lvert A \right\rvert ^{2}\left\lvert D_{k}  \right\rvert \left\lvert B \right\rvert  \right) ^{\frac{2}{3} }
,\]
so combining the previous results
\[
    k \left\lvert D_{k}  \right\rvert \ll \frac{1}{\left\lvert A \right\rvert } \cdot \left( \left\lvert A \right\rvert ^{2}\left\lvert D_{k}  \right\rvert \left\lvert B \right\rvert  \right) ^{\frac{2}{3} }
\]
or equivalently
\[
    k ^{3} \left\lvert D_{k}  \right\rvert \ll \left\lvert A \right\rvert \left\lvert B \right\rvert ^{2}
.\]
Recalling that \(k^{3} \left\lvert D_{k}  \right\rvert \approx E_3(A,B)\), the proof of (1) is complete.

We proceed with the proof of (2). Fix \(s \in (1,3)\). Interpolation using \holder's inequality gives
\[
    \sum _{x} \delta_{A,B} (x)^{s} = \sum _{x} \delta_{A,B} (x) ^{3 \cdot \frac{s-1}{2}  } \delta_{A,B} (x) ^{\frac{3- s }{2} } \leq \left( \sum _{x} \delta_{A,B} (x)^{3}  \right) ^{\frac{s-1}{2} } \left( \sum _{x} \delta_{A,B} (x) \right) ^{\frac{3-s}{2} }
.\]
Using the trivial result \( \sum _{x} \delta_{A,B} (x) = \left\lvert A \right\rvert \left\lvert B \right\rvert \) we have
\[
    E_{s} (A,B) \lesssim \left( \left\lvert A \right\rvert \left\lvert B \right\rvert ^{2} \right) ^{\frac{s-1}{2} } \left( \left\lvert A \right\rvert \left\lvert B \right\rvert  \right) ^{\frac{3 - s}{2} } = \left\lvert A \right\rvert \left\lvert B \right\rvert ^{\frac{s+1}{2} }
.\]
\end{proof}

\begin{notation}
For the rest of this section, for any set \(A \subset \mathbb{R} \), we let
\[
    A_{\lambda} := A \cap \left( \frac{A}{\lambda}  \right)
.\]
We do not use this notation after this section, and instead reserve subscripts for
enumeration of sets.
\end{notation}

For the general energy bound, we follow the framework of Rudnev-Stevens.
In particular we use the following result found in \cite{rudnev-stevens} (Proposition 1).

\begin{prop}\label{prop:rs-result}
Let \(A \subset \mathbb{R}^{+} \). Let \(\tau \in \mathbb{R} \) be so that, defining
\[
    S = \left\{ \lambda \in \frac{A}{A}  : r_{\frac{A}{A} } (\lambda) \in [\tau,2\tau) \right\} 
,\]
we have
\[
    E^{\times }(A) \approx \tau^{2} \left\lvert S  \right\rvert 
.\]

There exists \(S ' \subset S \) with \(\left\lvert S ' \right\rvert \geq \frac{1}{64} \cdot \left\lvert S \right\rvert \) such that,
for all \(\lambda \in S'\),
\[
    \left\lvert A A_{\lambda}  \right\rvert \gtrsim \frac{\left\lvert A \right\rvert ^{18}}{\left\lvert S  \right\rvert ^{\frac{1}{2} } \left\lvert AA \right\rvert ^{4} \left\lvert A + A \right\rvert ^{8}}
.\]
\end{prop}

We do not provide a proof of this result. Combining Lemma \ref{lem:solymosi-szt} and Proposition \ref{prop:rs-result} to obtain a result of the following type was
done by Rudnev-Stevens in \cite{rudnev-stevens}. This was later refined by Bloom in \cite{bloom} (Lemma 7).
What follows is not original work, but a restatement of the improvement of \cite{bloom}.

\begin{lemma}\label{lem:SP-szt}
Let \(A \subset \mathbb{R} ^{+}\). There is \(A_0 \subset A\) with \(\left\lvert A_0 \right\rvert > \frac{1}{2} \cdot \left\lvert A \right\rvert \) for which

(1) For all \(B \subset \mathbb{R} \), 
\[
    E_3(A_0,B) \lesssim \frac{\left\lvert B \right\rvert ^{2} \left\lvert AA \right\rvert ^{\frac{35}{2} } \left\lvert A + A \right\rvert ^{24}}{\left\lvert A \right\rvert ^{54}} 
\]

(2) For all \(B \subset \mathbb{R} \), and \(s \in (1,3)\),
\[
    E_{s} (A_0,B) \lesssim \left\lvert B \right\rvert ^{\frac{1 + s}{2} } \left\lvert A \right\rvert ^{\frac{1}{2} \cdot \left( 57 - 55s \right) } \left\lvert AA \right\rvert ^{\frac{35}{4} (s-1)} \left\lvert A+A \right\rvert ^{12 \left( s-1 \right) }
.\]
\end{lemma}

\begin{proof}
We first prove (1), from which (2) immediately follows by interpolation for \(E_{s} (A,B)\)
between \(E_3(A,B)\) and \(E_{1} (A,B) = \left\lvert A \right\rvert \left\lvert B \right\rvert \).

For the sake of clarity, let \(\Pi = AA\).

Suppose \(X \subset A\) with \(\left\lvert X \right\rvert  > \frac{1}{2} \cdot \left\lvert A \right\rvert \).
We apply Proposition \ref{prop:rs-result} to \(X\) instead of \(A\).
We show that, given \(X\) as above, Proposition \ref{prop:rs-result} gives rise to a
nonempty set \(X'\), a subset of some dilate of \(X\), for which
\(r_{\Pi / \Pi} (\lambda)\) is large for \(\lambda \in X'\). 

From this, we obtain
a bound on \(E_3(X',B)\). We then iterate this
general result, beginning with \(A\), obtaining a sequence of sets \(\left\{ A_{i} ' \right\} \),
and a bound on each of \(E_3(A_{i} ',B)\). The set \(\cup _{i} A_{i}' \) will be our \(A_0\).

A dyadic partitioning gives \(\beta \in \mathbb{R} \) so that, defining
\[
    S  : = \left\{ \lambda \in \frac{X}{X} : r_{\frac{X}{X} } (\lambda) \in [\beta,2\beta) \right\} 
\]
we have
\[
    E^{\times }(X) \approx \beta ^{2}\left\lvert S  \right\rvert 
.\]
Proposition \ref{prop:rs-result} applied to \(X\) yields \(S ' \subset S \) with \(\left\lvert S ' \right\rvert \geq \frac{1}{64} \cdot \left\lvert S  \right\rvert \) for which
for any \(\lambda \in S'\),
\begin{equation} \label{eq:xxlambda-bound}
    \left\lvert X X_{\lambda}  \right\rvert \gtrsim \frac{\left\lvert X \right\rvert ^{18}}{\left\lvert S \right\rvert ^{\frac{1}{2} } \left\lvert XX \right\rvert ^{4} \left\lvert X + X \right\rvert ^{8}}.
\end{equation}  
Since \(X \subset A\), \(\left\lvert A A_{\lambda}  \right\rvert \geq \left\lvert X X_{\lambda}  \right\rvert \). Additionally, \(\left\lvert X \right\rvert \gg \left\lvert A \right\rvert \)
and \(\left\lvert X + X \right\rvert \leq \left\lvert A+A \right\rvert \), \(\left\lvert XX \right\rvert \leq \left\lvert \Pi \right\rvert \).
Substituting all of these into (\ref{eq:xxlambda-bound}) we have that for any \(\lambda \in S'\),
\[
    \left\lvert A A_{\lambda}  \right\rvert \gtrsim  \frac{\left\lvert A \right\rvert ^{18}}{\left\lvert S \right\rvert ^{\frac{1}{2} } \left\lvert \Pi \right\rvert ^{4} \left\lvert A + A \right\rvert ^{8}} 
.\]

Observe the following truism. For \(\lambda \in S',~ a \in A,~ a_{\lambda} \in A_{\lambda} \). We have
\[
    \lambda = \frac{a \left( \lambda a_{\lambda}  \right) }{a a_{\lambda} } \in \frac{\Pi}{\Pi} 
.\]
As such, each distinct \(d \in A A_{\lambda} \) gives rise to the representation \((\lambda d, d) \in \Pi ^{2}\), and hence, \(\forall \lambda \in S '\),
\[
    r_{ \Pi / \Pi  } (\lambda) \geq \left\lvert A A_{\lambda}  \right\rvert \gtrsim \frac{\left\lvert A \right\rvert ^{18}}{\left\lvert S \right\rvert ^{\frac{1}{2} } \left\lvert \Pi \right\rvert ^{4} \left\lvert A + A \right\rvert ^{8}} 
.\]

We pass from elements in \(S'\) to elements in some dilate of \(X\) by pigeonholing.
Using the definition of \(S'\) and the definition of \(r_{X / X} \) respectively, we have
\[
    \beta \left\lvert S' \right\rvert \leq  \sum _{\lambda \in S'} r_{ X / X } (\lambda) = \sum _{\lambda \in S'} \sum _{x_0 \in X} \# \left\{ x \in X : \frac{x}{x_0} = \lambda \right\}   = \sum _{x_0 \in X} \# \left\{ x \in X : \frac{x}{x_0} \in S ' \right\}  
,\]
and hence there is \(x_0 \in X\) such that
\[
    \# \left\{ x \in X : \frac{x}{x_0} \in S'  \right\} \geq \frac{\beta \left\lvert S  \right\rvert }{64 \left\lvert X \right\rvert } 
.\]
Letting
\[
    X' = \left( \frac{X}{x_0}  \right) \cap S'
,\]
we see that \(\left\lvert X' \right\rvert \gg \frac{\beta \left\lvert S  \right\rvert }{\left\lvert X \right\rvert } \), and
for any \(x \in X'\)
\begin{equation} \label{eq:xprime-rep}
    r_{ \Pi / \Pi } \left( x  \right)  \gtrsim \frac{\left\lvert A \right\rvert ^{18}}{\left\lvert S \right\rvert ^{\frac{1}{2} } \left\lvert \Pi \right\rvert ^{4} \left\lvert A + A \right\rvert ^{8}}.
\end{equation}  

We have
\[
    \left\lvert X' \right\rvert \gg \frac{\beta \left\lvert S \right\rvert }{\left\lvert X \right\rvert }  \implies \beta \left\lvert S \right\rvert \ll \left\lvert X \right\rvert \left\lvert X' \right\rvert \leq \left\lvert A \right\rvert \left\lvert X' \right\rvert 
,\]
and therefore, substituting into the RHS of (\ref{eq:xprime-rep}),
\[
    \frac{\left\lvert A \right\rvert ^{18}}{\left\lvert S \right\rvert ^{\frac{1}{2} } \left\lvert \Pi \right\rvert ^{4} \left\lvert A + A \right\rvert ^{8}} \gg \frac{\left\lvert A \right\rvert ^{17} \beta \left\lvert S \right\rvert ^{\frac{1}{2} }}{\left\lvert X' \right\rvert \left\lvert \Pi \right\rvert ^{4} \left\lvert A + A \right\rvert ^{8}}
.\]
By the definition of \(\beta , S\), Cauchy-Schwarz, and \(\left\lvert X \right\rvert \gg \left\lvert A \right\rvert \), \(\left\lvert X X \right\rvert \leq \left\lvert \Pi \right\rvert \) respectively, we have
\[
    \beta ^{2} \left\lvert S \right\rvert \approx E^{\times }(X) \geq \frac{\left\lvert X \right\rvert ^{4}}{\left\lvert XX \right\rvert }  \gg \frac{\left\lvert A \right\rvert ^{4}}{\left\lvert \Pi \right\rvert } 
,\]
and hence
\[
    \frac{\left\lvert A \right\rvert ^{17} \beta \left\lvert S \right\rvert ^{\frac{1}{2} }}{\left\lvert X' \right\rvert \left\lvert \Pi \right\rvert ^{4} \left\lvert A + A \right\rvert ^{8}} \gg \frac{\left\lvert A \right\rvert ^{17} \left( \frac{\left\lvert A \right\rvert ^{4}}{\left\lvert \Pi \right\rvert }  \right) ^{\frac{1}{2} }}{\left\lvert X' \right\rvert \left\lvert \Pi \right\rvert ^{4} \left\lvert A + A \right\rvert ^{8}} = \frac{1}{\left\lvert X' \right\rvert } \cdot \frac{\left\lvert A \right\rvert ^{19}}{\left\lvert \Pi \right\rvert ^{\frac{9}{2} }\left\lvert A+A \right\rvert ^{8}}  
.\]

Substituting into (\ref{eq:xprime-rep}) gives that for any \(x \in X'\) we have
\begin{equation} \label{eq:xprime-rep-final}
    r_{ \Pi / \Pi } \left( x  \right)  \gtrsim \frac{1}{\left\lvert X' \right\rvert } \cdot \frac{\left\lvert A \right\rvert ^{19}}{\left\lvert \Pi \right\rvert ^{\frac{9}{2} } \left\lvert A + A \right\rvert ^{8}} .
\end{equation}

We see that \(\left\lvert X' \right\rvert \geq 1\), or else
\[
    \beta \left\lvert S \right\rvert \ll \left\lvert X \right\rvert \leq \left\lvert A \right\rvert 
,\]
and hence, dividing from
\[
    \beta^{2} \left\lvert S \right\rvert \approx E^{\times }(X) \geq \frac{\left\lvert X \right\rvert ^{4}}{\left\lvert XX \right\rvert } \gg \frac{\left\lvert A \right\rvert ^{4}}{\left\lvert \Pi \right\rvert } 
,\]
we see that
\[
    \beta \gtrsim \frac{\left\lvert A \right\rvert ^{3} }{\left\lvert \Pi \right\rvert } 
.\]
Using the trivial bound \(\beta \leq \left\lvert A \right\rvert \) gives
\(\left\lvert \Pi \right\rvert \gtrsim \left\lvert A \right\rvert ^{2}\), and hence Lemma \ref{lem:SP-szt} part 1 holds trivially
upon taking \(A_0 = A\).

We proceed by obtaining an energy bound for \(X'\).
Fix an arbitrary \(B \subset \mathbb{R} \) finite.

By a dyadic partitioning, there is \(k \in \mathbb{N} \) and
\[
    D^{(k)} = \left\{ d \in X'  - B  : \delta_{X',B} (d) \in [k,2k) \right\}
\]
such that
\[
    E_{3} (X',B)\approx k ^{3} \left\lvert D^{(k)} \right\rvert 
.\]

By the definition of \(D^{(k)}\),
\[
    \sum _{d \in D^{(k)}} \delta_{X',B} (d) \geq k \left\lvert D^{(k)} \right\rvert 
.\]
We have
\[
    \sum _{d \in D^{(k)}} \delta_{X',B} (d) = \sum _{x \in X'} \sigma_{B,D^{(k)}} (x)
,\]
and hence, by (\ref{eq:xprime-rep-final}), we have
\begin{equation} \label{eq:lem-lower-bound}
    \sum _{x \in X'} \sigma_{B,D^{(k)}} (x) r_{ \Pi / \Pi} (x) \gtrsim k \left\lvert D^{(k)} \right\rvert \cdot \frac{1}{\left\lvert X' \right\rvert } \cdot \frac{\left\lvert A \right\rvert ^{19}}{\left\lvert \Pi \right\rvert ^{\frac{9}{2} } \left\lvert A + A \right\rvert ^{8}}  .
\end{equation}  

Trivially we have
\[
    \sum _{x \in X'} \sigma_{B,D^{(k)}} (x) r_{ \Pi / \Pi} (x) \leq \sum _{x} \sigma_{B,D^{(k)}} (x) r_{ \Pi / \Pi} (x)
.\]
See that
\[
    \sum _{x} \sigma_{B,D^{(k)}} (x) r_{ \Pi / \Pi} (x)
\]
counts solutions to
\begin{equation} \label{eq:point-line-system}
    \frac{1}{\pi_2} \cdot \pi_1 - d = b,~ \pi_{i} \in \Pi,~ b \in B,~ d \in D^{(k)}.
\end{equation}
Solutions to (\ref{eq:point-line-system}) are precisely the incidences of the point set
\(\Pi \times B\) with the system of lines
\[
    \ell(x) = \frac{x}{\pi} - d ,~ \pi \in \Pi ,~ d \in D^{(k)}
.\]
Since \(\Pi \subset \mathbb{R} _{>0} \), the slopes of
these lines are finite nonzero real numbers, so applying Lemma \ref{lem:solymosi-szt} gives
\[
    \sum _{x} \sigma_{B,D^{(k)}} (x) r_{ \Pi / \Pi} (x) \ll \left( \left\lvert \Pi \right\rvert ^{2} \left\lvert B \right\rvert \left\lvert D^{(k)} \right\rvert  \right) ^{\frac{2}{3} } + \left\lvert \Pi \right\rvert \left\lvert D^{(k)} \right\rvert 
.\]
Using the trivial bound \(\left\lvert D^{(k)} \right\rvert \leq  \left\lvert A \right\rvert \left\lvert B \right\rvert \) we have
\[
    \left\lvert \Pi \right\rvert \left\lvert D^{(k)} \right\rvert \ll \left( \left\lvert \Pi \right\rvert ^{2} \left\lvert B \right\rvert \left\lvert D^{(k)} \right\rvert  \right) ^{\frac{2}{3} }
,\]
so combining with (\ref{eq:lem-lower-bound}) gives
\[
    k \left\lvert D^{(k)} \right\rvert \cdot \frac{1}{\left\lvert X' \right\rvert } \cdot \frac{\left\lvert A \right\rvert ^{19}}{\left\lvert \Pi \right\rvert ^{\frac{9}{2} } \left\lvert A + A \right\rvert ^{8}} \lesssim \sum _{x} \sigma_{B,D^{(k)}} (x) r_{ \Pi / \Pi} (x) \ll \left( \left\lvert \Pi \right\rvert ^{2} \left\lvert B \right\rvert \left\lvert D^{(k)} \right\rvert  \right) ^{\frac{2}{3} }
,\]
or equivalently
\[
    k^{3}  \left\lvert D^{(k)} \right\rvert \lesssim \left\lvert X' \right\rvert ^{3} \cdot \frac{\left\lvert B \right\rvert ^{2} \left\lvert \Pi \right\rvert ^{\frac{35}{2} }\left\lvert A+A \right\rvert ^{24}}{\left\lvert A \right\rvert ^{57}} 
.\]
Recall that \(k ^{3} \left\lvert D^{(k)} \right\rvert \approx E_3(X',B)\).

Having found the desired bound on \(E_3(X',B)\) we proceed with iteration.
Let \(A_1 = A\). 
If \(\left\lvert A_{j} \right\rvert  > \frac{1}{2} \cdot \left\lvert A \right\rvert  \), obtain \(A_{j+1} \) from \(A_{j} \) by applying the above argument
with \(X = A_{j} \), yielding some \(a_{j}  \in A\) and a subset \(A_{j} ' \subset A_{j} \) for which we have
\[
    \forall a \in A_{j} ' ,~ r_{ \Pi / \Pi } \left( \frac{a}{a_{j} }   \right)  \gtrsim \frac{1}{\left\lvert A_{j} ' \right\rvert } \cdot \frac{\left\lvert A \right\rvert ^{19}}{\left\lvert \Pi \right\rvert ^{\frac{9}{2} } \left\lvert A + A \right\rvert ^{8}} 
,\]
and hence for any \(B \subset \mathbb{R} \),
\[
    E_3(A_{j} ',B) = E_3 \left( \frac{A_{j} '}{a_{j} } , \frac{B}{a_{j} }    \right) \lesssim \left\lvert A_{j} ' \right\rvert ^{3} \cdot \frac{\left\lvert B \right\rvert ^{2} \left\lvert \Pi \right\rvert ^{\frac{35}{2} }\left\lvert A+A \right\rvert ^{24}}{\left\lvert A \right\rvert ^{57}} 
.\]
Let \(A_{j+1} = A_{j} \setminus A_{j} ' \). Once \(\left\lvert A_{n}  \right\rvert \leq \frac{1}{2} \cdot \left\lvert A \right\rvert     \), which will occur in a finite number of steps since \(\left\lvert A_{j} '  \right\rvert \geq 1\),
let
\[
    A_0 = \bigsqcup  _{j=1} ^{n-1}A_{j} ' = A \setminus A_{n} 
.\]
We have \(\left\lvert A_0 \right\rvert > \frac{1}{2} \cdot \left\lvert A \right\rvert \) and, by Minkowski's inequality, for any \(B \subset \mathbb{R} \),
\[
    E_3(A_0,B)^{\frac{1}{3} } = \left\lVert \delta_{A_0,B}  \right\rVert _{3} = \left\lVert \sum _{j =1} ^{n-1} \delta_{A_{j}' ,B}  \right\rVert _{3} \leq \sum _{j=1} ^{n-1} \left\lVert \delta_{A_{j}' ,B}  \right\rVert _{3} \lesssim \frac{\left\lvert B \right\rvert ^{\frac{2}{3} } \left\lvert \Pi \right\rvert ^{\frac{35}{6}  } \left\lvert A+ A \right\rvert ^{8}}{\left\lvert A \right\rvert ^{19}} \sum _{j=1} ^{n-1} \left\lvert A_{j} ' \right\rvert 
.\]
Seeing that \( \sum _{j=1} ^{n-1}\left\lvert A_{j} ' \right\rvert \leq \left\lvert A \right\rvert   \) gives
\[
    E_3(A_0,B)\lesssim \frac{\left\lvert B \right\rvert^{2}  \left\lvert \Pi \right\rvert ^{\frac{35}{2} } \left\lvert A+A \right\rvert ^{24}}{\left\lvert A \right\rvert ^{54}} 
\]
as desired. Part (1) is complete.

For part (2), fix \(s \in (1,3)\) and use \holder's inequality to get
\[
    \sum _{x} \delta_{A_0,B} (x)^{s} = \sum _{x} \delta_{A_0,B} (x) ^{3 \cdot \frac{s-1}{2}  } \delta_{A_0,B} (x) ^{\frac{3- s }{2} } \leq \left( \sum _{x} \delta_{A_0,B} (x)^{3}  \right) ^{\frac{s-1}{2} } \left( \sum _{x} \delta_{A_0,B} (x) \right) ^{\frac{3-s}{2} }
.\]
Seeing that, by part (1) and the trivial result \( \sum _{x} \delta_{A_0,B} (x) = \left\lvert A_0 \right\rvert \left\lvert B \right\rvert \),
we have
\[
    \left( \sum _{x} \delta_{A_0,B} (x)^{3}  \right) ^{\frac{s-1}{2} } \left( \sum _{x} \delta_{A_0,B} (x) \right) ^{\frac{3-s}{2} } \lesssim \left( \frac{\left\lvert B \right\rvert^{2}  \left\lvert \Pi \right\rvert ^{\frac{35}{2} } \left\lvert A+A \right\rvert ^{24}}{\left\lvert A \right\rvert ^{54}}  \right) ^{\frac{s-1}{2} }\left( \left\lvert A_0 \right\rvert \left\lvert B \right\rvert  \right) ^{\frac{3-s}{2} }
,\]
and so
\[
    E_{s} (A_0,B) \lesssim \left\lvert B \right\rvert ^{\frac{1 + s}{2} } \left\lvert A \right\rvert ^{\frac{1}{2} \cdot \left( 57 - 55s \right) } \left\lvert \Pi \right\rvert ^{\frac{35}{4} (s-1)} \left\lvert A+A \right\rvert ^{12 \left( s-1 \right) }
.\]

\end{proof}

\begin{remark}
We now stop using the notation \(A_{\lambda} = A \cap \left(  \frac{A}{\lambda}  \right) \), and instead
reserve subscripts for enumeration of sets.
\end{remark}

\bigskip

\section{Proof of Theorem \ref{thm-SP}}\label{section:SP}

Let \(A \subset \mathbb{R} \) be a sufficiently large finite set (in the sense of Lemma \ref{lem:sum-proj}).

We have WLOG \(A \subset \mathbb{R} ^{+}\). If not, we break \(A\setminus \left\{ 0 \right\} \) into positive and negative parts
as \(A \setminus \left\{ 0 \right\} = A^{+} \sqcup A^{-} \) where \(A^{+} , -A^{-} \subset \mathbb{R} ^{+}\). Either \(\left\lvert A^{+} \right\rvert  > \frac{1}{3} \cdot  \left\lvert A \right\rvert \) or \(\left\lvert A ^{-}\right\rvert  > \frac{1}{3} \cdot \left\lvert A \right\rvert  \).
Suppose it is the first.
We have
\[
    \max \left( \left\lvert A+A \right\rvert , \left\lvert AA \right\rvert  \right) \geq \max \left( \left\lvert A^{+ } + A^{+ } \right\rvert , \left\lvert A^{+ }A^{+ } \right\rvert  \right) 
,\]
so we may apply the following argument with \(A^{+}\) in place of \(A\) and finish by using \(\left\lvert A^{+} \right\rvert \gg \left\lvert A \right\rvert \).

From now on we assume WLOG \(A \subset \mathbb{R} ^{+}\).
Lemma \ref{lem:SP-szt} gives \(A_0 \subset A\) with \(\left\lvert A_0 \right\rvert > \frac{1}{2} \cdot \left\lvert A \right\rvert \) with some energy bounds.
Apply Lemma \ref{lem:sum-proj} to the set \(A_0\), as opposed to \(A\), to give \(B \subset A_0 \subset A\) with
\(\left\lvert B \right\rvert \geq \frac{1}{2} \cdot \left\lvert A_0 \right\rvert \), and \(\Delta\), \(P_{\Delta} \) for which
\[
    \Delta^{\frac{12}{7} } \left\lvert P_{\Delta}  \right\rvert \approx E_{\frac{12}{7} } \left( R_{A_0} (B) \right) \approx E_{\frac{12}{7} } (B) 
\]
and
\[
    \Delta ^{2} \left\lvert P_{\Delta}  \right\rvert ^{2} \left\lvert B \right\rvert ^{2} \ll E_3(B) \cdot \sum _{p \in P_{\Delta} } \delta_{P_{A_0} (B)} (p)
.\]
To ease notation, call \(P_{A_0} (B)= S\).

Firstly, see that
\[
    \sum _{p \in P_{\Delta} } \delta_{S} (p) = \sum _{x \in S} \sigma_{S, P_{\Delta} }(x) 
,\]
as they both count solutions to
\[
    s_1-s_2  = p_{\Delta} ,~ s_{i}  \in S ,~ p_{\Delta} \in P_{\Delta} 
.\]
See that, by the definition of \(S = P_{A_0} (B)\),
\[
    \frac{\left\lvert B \right\rvert ^{2} }{\left\lvert B + B \right\rvert} \sum _{x \in S} \sigma_{S, P_{\Delta} } (x) \lesssim \sum _{x \in S} \sigma_{B} (x) \sigma_{S,P_{\Delta} }(x) 
.\]
Rearranging again, see that
\[
    \sum _{x} \sigma_{B} (x) \sigma_{S , P_{\Delta} } (x) = \sum _{x} \delta_{B,P_{\Delta} } (x) \delta_{S, B} (x)
,\]
as they both count solutions to
\[
    b_1 + b_2 = s + p_{\Delta} ,~ b_{i} \in B,~ s  \in S ,~ p_{\Delta} \in P_{\Delta} 
.\]

By \holder's inequality,
\[
    \sum _{x} \delta_{B,P_{\Delta} } (x) \delta_{S,B} (x) \leq E_{\frac{3}{2} } (B,P_{\Delta} )^{\frac{2}{3} } E_{3} (S,B)^{\frac{1}{3} }
,\]
so combining the above results yields
\[
    \sum _{p \in P_{\Delta} } \delta_{S} (p) \lesssim \frac{\left\lvert B+B \right\rvert }{\left\lvert B \right\rvert ^{2}}\cdot  E_{\frac{3}{2} } (B,P_{\Delta} )^{\frac{2}{3} } E_{3} (S,  B)^{\frac{1}{3} }
.\]
Substituting into the original inequality gives
\begin{equation} \label{eq:SP-from-lemma}
    \Delta ^{2}\left\lvert P_{\Delta}  \right\rvert ^{2} \left\lvert B \right\rvert ^{2} \ll E_3(B) \cdot \frac{\left\lvert B+B \right\rvert }{\left\lvert B \right\rvert ^{2}} \cdot E_{\frac{3}{2} } (B,P_{\Delta} )^{\frac{2}{3} } E_3(S,B)^{\frac{1}{3} }.
\end{equation}

Both \(E_3(B)\) and \(E_{\frac{3}{2} } (B,P_{\Delta} )\) can be bounded easily using Lemma \ref{lem:SP-szt},
so we proceed with the final term
\(E_{3} (S,B)^{\frac{1}{3} }\). Firstly, let \(D_{i} = \left\{ x : \delta_{B}  (x) \in [2^{i},2^{i+1}) \right\} \),
and let \(\tau_{x} :\mathbb{R}  \to \mathbb{R} \) be the ``translation'' function, so
\[
    \tau_{x} (y) = y + x
.\]
Seeing that, for any \(x \in S-B\), by the definition of \(S = P_{A_0} (B)\),
\[
    \delta_{S,B} (x) \lesssim  \frac{\left\lvert B+B \right\rvert }{\left\lvert B \right\rvert ^{2}} \cdot r_{B + B - B} (x) 
,\]
and hence we have
\[
   E_{3} (S,B)^{\frac{1}{3} } = \left\lVert \delta_{S,B}  \right\rVert _{3}  \lesssim  \frac{\left\lvert B + B\right\rvert }{\left\lvert B \right\rvert ^{2}} \left\lVert r_{B + B-B} \right\rVert _{3}
.\]

Using the definitions of \(\tau\) and \(r_{B + B - B} \), we have
\[
    \left\lVert r_{B + B - B}  \right\rVert _{3}  = \left\lVert \sum _{b_1,b_2 \in B}  1_{B} \circ \tau_{b_1 -b_2}   \right\rVert _{3} 
.\]
Partitioning the sum over \(B-B\) and \(D_{i} \) respectively gives
\[
    \left\lVert \sum _{b_1,b_2 \in B}  1_{B} \circ \tau_{b_1 -b_2}   \right\rVert _{3} = \left\lVert \sum _{x \in B-B} 1_{B} \circ \tau_{x} \cdot \delta_{B} (x) \right\rVert   _{3} = \left\lVert \sum _{i \leq \log _{2} \left\lvert B \right\rvert } \sum _{x \in D_{i} } 1_{B} \circ \tau_{x} \cdot \delta_{B} (x) \right\rVert _{3} 
.\]
By Minkowski's inequality and definition of \(D_{i} \) respectively,
\begin{align*}
    \left\lVert \sum _{i \leq \log _{2} \left\lvert B \right\rvert } \sum _{x \in D_{i} } 1_{B} \circ \tau_{x} \cdot \delta_{B} (x) \right\rVert _{3} & \leq \sum _{i \leq \log _{2} \left\lvert B \right\rvert } \left\lVert \sum _{x \in D_{i} } 1_{B} \circ \tau_{x} \cdot \delta_{B} (x) \right\rVert _{3} \\
    & \leq \sum _{i \leq \log _{2} \left\lvert B \right\rvert } 2^{i+1} \left\lVert \sum _{x \in D_{i} } 1_{B} \circ \tau_{x}  \right\rVert _{3} 
\end{align*}

By the definition of \(\delta_{B,D_{i} } \),
\[
    \left[ \sum _{x \in D_{i} } 1_{B} \circ \tau_{x} \right] (y) = \delta_{B,D_{i} } (y)
,\]
and hence we have
\[
    \sum _{i \leq \log _{2} \left\lvert B \right\rvert } 2^{i+1} \left\lVert \sum _{x \in D_{i} } 1_{B} \circ \tau_{x}  \right\rVert _{3} = \sum _{i \leq \log _{2} \left\lvert B \right\rvert } 2^{i+1} \left\lVert \delta_{B,D_{i} }  \right\rVert _{3} 
.\]
Thus far, we have demonstrated that
\begin{equation} \label{eq:bloom-arg-SP-case}
\left\lVert \delta_{S,B}  \right\rVert _{3}  \lesssim \frac{\left\lvert B+B \right\rvert }{\left\lvert B \right\rvert ^{2}} \sum _{i \leq \log _{2} \left\lvert B \right\rvert } 2^{i+1} \left\lVert \delta_{B,D_{i} }  \right\rVert _{3}   ,
\end{equation}
which we record now for later use. As \(E_3(B,D_{i} )\leq E_{3} (A_0,D_{i} )\), we apply Lemma \ref{lem:SP-szt}
to get
\[
    \left\lVert \delta_{B,D_{i} }  \right\rVert _{3} \leq  \left\lVert \delta_{A_0,D_{i} }  \right\rVert _{3}  \lesssim \frac{\left\lvert D_{i}  \right\rvert ^{\frac{2}{3} } \left\lvert AA \right\rvert ^{\frac{35}{6} } \left\lvert A + A \right\rvert ^{8}}{\left\lvert A \right\rvert ^{18}} 
.\]
Combining with (\ref{eq:bloom-arg-SP-case}) gives
\[
    \left\lVert \delta_{S,B}  \right\rVert _{3}  \lesssim \frac{\left\lvert B + B \right\rvert }{\left\lvert B \right\rvert ^{2}} \cdot \frac{ \left\lvert AA \right\rvert ^{\frac{35}{6} } \left\lvert A + A \right\rvert ^{8}}{\left\lvert A \right\rvert ^{18}} \cdot \sum _{i \leq \log _{2} \left\lvert B \right\rvert } 2^{i+1}\left\lvert D_{i}  \right\rvert ^{\frac{2}{3} }
\]
Pigeonholing gives a \(j \in \mathbb{N}  \) so that
\[
    \sum _{i \leq \log _{2} \left\lvert B \right\rvert } 2^{i+1} \left\lvert D_{i}  \right\rvert ^{\frac{2}{3} } \lesssim 2^{j} \left\lvert D_{j}  \right\rvert ^{\frac{2}{3} } \leq \left(  \sum _{x \in D_{j} } \delta_{B} (x) ^{\frac{3}{2} } \right) ^{\frac{2}{3} } \leq  E_{\frac{3}{2} } (B) ^{\frac{2}{3} }
.\]

Interpolating using \holder's inequality gives
\[
    E_{\frac{3}{2} } (B)^{\frac{2}{3} } \leq \left\lvert B \right\rvert ^{\frac{2}{5} }  E_{\frac{12}{7} } (B) ^{\frac{7}{15} }
,\]
and hence
\[
    E_{3} (S,B)^{\frac{1}{3} }\lesssim \frac{\left\lvert B+B \right\rvert }{\left\lvert B \right\rvert ^{2}} \cdot  \frac{ \left\lvert AA \right\rvert ^{\frac{35}{6} } \left\lvert A + A \right\rvert ^{8}}{\left\lvert A \right\rvert ^{18}}\cdot \left\lvert B \right\rvert ^{\frac{2}{5} }E_{\frac{12}{7} } (B)^{\frac{7}{15} }
\]
or, using \(\left\lvert B \right\rvert \gg \left\lvert A \right\rvert \) and \(\left\lvert B+B \right\rvert \leq \left\lvert A+A \right\rvert \),
\begin{equation} \label{eq:third-term-resolved}
   E_3(S,B)^{\frac{1}{3} }\lesssim \frac{\left\lvert AA \right\rvert ^{\frac{35}{6} }\left\lvert A+A \right\rvert ^{9}}{\left\lvert A \right\rvert ^{\frac{98}{5} }} \cdot E_{\frac{12}{7} } (B)^{\frac{7}{15} }.
\end{equation}

Finally, two applications of Lemma \ref{lem:SP-szt} give
\begin{equation} \label{eq:first-term-resolved}
E_3(B) \leq E_3(A_0)\lesssim \frac{\left\lvert AA \right\rvert ^{\frac{35}{2} } \left\lvert A + A \right\rvert ^{24}}{\left\lvert A \right\rvert ^{52}} 
\end{equation}
and
\begin{equation} \label{eq:second-term-resolved}
E_{\frac{3}{2} } (B,P_{\Delta} )^{\frac{2}{3} } \leq E_{\frac{3}{2} } (A_0,P_{\Delta} )^{\frac{2}{3} }\lesssim \frac{\left\lvert P_{\Delta}  \right\rvert ^{\frac{5}{6}  }  \left\lvert AA \right\rvert ^{\frac{35}{12} } \left\lvert A+A \right\rvert ^{4}}{\left\lvert A \right\rvert ^{\frac{17}{2} }} 
\end{equation}

and substituting (\ref{eq:third-term-resolved}), (\ref{eq:first-term-resolved}), and (\ref{eq:second-term-resolved}) into 
(\ref{eq:SP-from-lemma}) gives
\[\hspace*{-10mm}
    \Delta ^{2} \left\lvert P_{\Delta}  \right\rvert ^{2} \left\lvert B \right\rvert ^{2} \lesssim  \frac{\left\lvert AA \right\rvert ^{\frac{35}{2} } \left\lvert A + A \right\rvert ^{24}}{\left\lvert A \right\rvert ^{52}} \cdot \frac{\left\lvert B+B \right\rvert }{\left\lvert B \right\rvert ^{2}} \cdot \frac{\left\lvert P_{\Delta}  \right\rvert ^{\frac{5}{6}  }  \left\lvert AA \right\rvert ^{\frac{35}{12} } \left\lvert A+A \right\rvert ^{4}}{\left\lvert A \right\rvert ^{\frac{17}{2} }} \cdot \frac{\left\lvert AA \right\rvert ^{\frac{35}{6} }\left\lvert A+A \right\rvert ^{9}}{\left\lvert A \right\rvert ^{\frac{98}{5} }} \cdot E_{\frac{12}{7} } (B)^{\frac{7}{15} }
.\]
Isolating the \(\Delta, \left\lvert P_{\Delta}  \right\rvert \) terms and using \(\left\lvert B \right\rvert \gg \left\lvert A \right\rvert \) and \(\left\lvert B + B \right\rvert \leq \left\lvert A + A  \right\rvert \), gives
\[
    \Delta^{2} \left\lvert P_{\Delta}  \right\rvert ^{\frac{7}{6} } \lesssim  \frac{\left\lvert AA \right\rvert ^{\frac{105}{4}} \left\lvert A+A \right\rvert^{38} }{\left\lvert A \right\rvert^{\frac{841}{10} } } \cdot E_{\frac{12}{7} } (B)^{\frac{7}{15} }
.\]

By the definition of \(\Delta,P_{\Delta} \),
\[
    \Delta ^{2}\left\lvert P_{\Delta}  \right\rvert ^{\frac{7}{6} } = \left( \Delta^{\frac{12}{7} } \left\lvert P_{\Delta}  \right\rvert  \right) ^{\frac{7}{6} } \approx E_{\frac{12}{7} } (B) ^{\frac{7}{6} }
,\]
substituting and simplifying gives
\[
    E_{\frac{12}{7} } (B)\lesssim \frac{\left\lvert AA \right\rvert ^{\frac{75}{2} } \left\lvert A+A \right\rvert ^{\frac{380}{7} }}{\left\lvert A \right\rvert ^{\frac{841}{7} }} 
.\]

Interpolating for \(E(B)\) using \holder's inequality, and using (\ref{eq:first-term-resolved}) gives
\begin{align*}
E(B) & \leq  E_{\frac{12}{7} } (B)^{\frac{7}{9} }  E_3(B)  ^{\frac{2}{9} } \\
& \lesssim \left( \frac{\left\lvert AA \right\rvert ^{\frac{75}{2} } \left\lvert A+A \right\rvert ^{\frac{380}{7} }}{\left\lvert A \right\rvert ^{\frac{841}{7} }}  \right) ^{\frac{7}{9} } \left( \frac{\left\lvert AA \right\rvert ^{\frac{35}{2} } \left\lvert A + A \right\rvert ^{24}}{\left\lvert A \right\rvert ^{52}}  \right) ^{\frac{2}{9} }\\
& = \frac{\left\lvert AA \right\rvert ^{\frac{595}{18}} \left\lvert A+A \right\rvert^{\frac{428}{9}} }{\left\lvert A \right\rvert ^{105}} 
\end{align*}

By Cauchy-Schwarz we have \(E(B) \geq \frac{\left\lvert B \right\rvert ^{4}}{\left\lvert B+B \right\rvert } \gg \frac{\left\lvert A \right\rvert ^{4}}{\left\lvert A+A \right\rvert }  \), and hence
\[
    \frac{\left\lvert A \right\rvert ^{4}}{\left\lvert A + A \right\rvert } \lesssim \frac{\left\lvert AA \right\rvert ^{\frac{595}{18}} \left\lvert A+A \right\rvert^{\frac{428}{9}} }{\left\lvert A \right\rvert ^{105}} \implies \max (\left\lvert A+A \right\rvert , \left\lvert AA \right\rvert )\gtrsim \left\lvert A \right\rvert ^{\frac{4}{3}  + \frac{10}{4407}}
,\]
from which Theorem \ref{thm-SP} follows.

\bigskip

\section{Proof of Theorem \ref{thm-CSUM}}\label{section:CSUM}
Let \(A\) be a sufficiently large finite set (in the sense of Lemma \ref{lem:sum-proj}), which is convex.
Apply Lemma \ref{lem:sum-proj} to the set \(A\) to give \(B \subset A\) with
\(\left\lvert B \right\rvert \geq \frac{1}{2} \cdot \left\lvert A \right\rvert \), and \(\Delta\), \(P_{\Delta} \) for which
\[
    \Delta^{\frac{12}{7} } \left\lvert P_{\Delta}  \right\rvert \approx E_{\frac{12}{7} } \left( R_{A} (B) \right) \approx E_{\frac{12}{7} } (B) 
\]
and
\begin{equation} \label{eq:CSUM-from-lemma}
    \Delta ^{2} \left\lvert P_{\Delta}  \right\rvert ^{2} \left\lvert B \right\rvert ^{2} \ll E_3(B) \cdot \sum _{p \in P_{\Delta} } \delta_{P_{A} (B)} (p)
\end{equation}
To ease notation, call \(P_{A} (B)= S\). This should cause no confusion with how \(S\) is defined in the previous section,
as the \(A_0\) of the previous section has \(\left\lvert A_0 \right\rvert \asymp \left\lvert A \right\rvert \).

We proceed almost identically as in the proof of Theorem \ref{thm-SP} in the previous section,
the only difference being the use of Lemma \ref{lem:convex-szt} in place of Lemma \ref{lem:SP-szt}.

An argument identical to that of the previous section gives
\[
    \sum _{p\in P_{\Delta} } \delta_{S} (p) \lesssim \frac{\left\lvert B+B \right\rvert }{\left\lvert B \right\rvert ^{2}} \cdot E_{\frac{3}{2} } (B,P_{\Delta} )^{\frac{2}{3} } E_{3} (S,B)^{\frac{1}{3} }
,\]
and hence
\begin{equation} \label{eq:convex-to-substitute}
    \Delta^{2}\left\lvert P_{\Delta}  \right\rvert ^{2} \left\lvert B \right\rvert ^{2} \lesssim E_3(B) \cdot \frac{\left\lvert B+B \right\rvert }{\left\lvert B \right\rvert ^{2}}\cdot  E_{\frac{3}{2} } (B,P_{\Delta} )^{\frac{2}{3} } E_{3} (S,  B)^{\frac{1}{3} }.
\end{equation}
Letting \(D_{i} = \left\{ x : \delta_{B} (x) \in [2^{i},2^{i+1}) \right\} \), by the exact same argument as in the previous section,
\[
    \left\lVert \delta_{S,B}  \right\rVert _{3}   \lesssim \frac{\left\lvert B+B \right\rvert }{\left\lvert B \right\rvert ^{2}} \sum _{i \leq \log _{2} \left\lvert B \right\rvert } 2^{i+1} \left\lVert \delta_{B,D_{i} }  \right\rVert _{3} 
.\]
Lemma \ref{lem:convex-szt} gives
\[
    \left\lVert \delta_{B,D_{i} }  \right\rVert _{3} \leq \left\lVert \delta_{A,D_{i} }  \right\rVert _{3} \lesssim \left\lvert A \right\rvert ^{\frac{1}{3} } \left\lvert D_{i}  \right\rvert ^{\frac{2}{3} }
,\]
so
\[
    \left\lVert \delta_{S,B}  \right\rVert _{3} \lesssim \frac{\left\lvert B+B \right\rvert }{\left\lvert B \right\rvert ^{2}} \cdot \left\lvert A \right\rvert ^{\frac{1}{3} }\cdot \sum _{i \leq \log _{2} \left\lvert B \right\rvert } 2^{i+1}\left\lvert D_{i}  \right\rvert ^{\frac{2}{3} }
.\]
Pigeonholing gives a \(j \in \mathbb{N}  \) so that
\[
    \sum _{i \leq \log _{2} \left\lvert B \right\rvert } 2^{i+1} \left\lvert D_{i}  \right\rvert ^{\frac{2}{3} } \lesssim 2^{j} \left\lvert D_{j}  \right\rvert ^{\frac{2}{3} } \leq \left(  \sum _{x \in D_{j} } \delta_{B} (x) ^{\frac{3}{2} } \right) ^{\frac{2}{3} } \leq  E_{\frac{3}{2} } (B) ^{\frac{2}{3} }
.\]

Interpolating using \holder's inequality gives
\[
    E_{\frac{3}{2} } (B)^{\frac{2}{3} } \leq \left\lvert B \right\rvert ^{\frac{2}{5} }  E_{\frac{12}{7} } (B) ^{\frac{7}{15} }
,\]
and hence, using \(\left\lvert B \right\rvert \gg \left\lvert A \right\rvert \) and \(\left\lvert B+B \right\rvert \leq \left\lvert A+A \right\rvert \)
\begin{equation} \label{eq:convex-first-term}
    E_{3} (S,B)^{\frac{1}{3} }\lesssim \frac{\left\lvert A+A \right\rvert }{\left\lvert A \right\rvert ^{\frac{19}{15} }} \cdot E_{\frac{12}{7} } (B)^{\frac{7}{15} }.
\end{equation}
Lemma \ref{lem:convex-szt} gives
\begin{equation} \label{eq:convex-second-term}
    E_3(B) \leq E_3(A) \lesssim \left\lvert A \right\rvert ^{3} 
\end{equation}
and
\begin{equation} \label{eq:convex-third-term}
    E_{\frac{3}{2} }  (B,P_{\Delta} )^{\frac{2}{3} } \leq E_{\frac{3}{2} } (A,P_{\Delta} )^{\frac{2}{3} } \lesssim \left\lvert A \right\rvert ^{\frac{2}{3} }\left\lvert P_{\Delta} \right\rvert ^{\frac{5}{6} }.
\end{equation}  

Substituting (\ref{eq:convex-first-term}), (\ref{eq:convex-second-term}), and (\ref{eq:convex-third-term}) into (\ref{eq:convex-to-substitute})
gives
\[
    \Delta^{2} \left\lvert P_{\Delta}  \right\rvert ^{2} \left\lvert B \right\rvert ^{2} \lesssim \left\lvert A \right\rvert ^{3}  \cdot \frac{\left\lvert A+A \right\rvert }{\left\lvert A \right\rvert ^{2}} \cdot \left\lvert A \right\rvert ^{\frac{2}{3} }\left\lvert P_{\Delta}  \right\rvert ^{\frac{5}{6} }\cdot \frac{\left\lvert A+A \right\rvert }{\left\lvert A \right\rvert^{\frac{19}{15} }}\cdot E_{\frac{12}{7} } (B)^{\frac{7}{15} } 
\]
and hence
\[
    \Delta^{2}\left\lvert P_{\Delta}  \right\rvert ^{\frac{7}{6} } \lesssim \frac{\left\lvert A+A \right\rvert ^{2}}{\left\lvert A \right\rvert ^{\frac{8}{5} }}  \cdot E_{\frac{12}{7} } (B)^{\frac{7}{15} }
.\]

By the definition of \(\Delta, P_{\Delta} \) we have
\[
    \Delta ^{2}\left\lvert P_{\Delta}  \right\rvert ^{\frac{7}{6} } = \left( \Delta^{\frac{12}{7} } \left\lvert P_{\Delta}  \right\rvert  \right) ^{\frac{7}{6} } \approx E_{\frac{12}{7} } (B) ^{\frac{7}{6} }
,\]
upon which substituting and simplifying gives
\[
    E_{\frac{12}{7} } (B)\lesssim \frac{\left\lvert A+A \right\rvert ^{\frac{20}{7} }}{\left\lvert A \right\rvert ^{\frac{16}{7} }} 
.\]

Interpolating for \(E(B)\) using \holder's inequality, and using (\ref{eq:convex-second-term}) gives
\begin{align*}
E(B) & \leq  E_{\frac{12}{7} } (B)^{\frac{7}{9} }  E_3(B)  ^{\frac{2}{9} } \\
& \lesssim \left( \frac{\left\lvert A+A \right\rvert ^{\frac{20}{7} }}{\left\lvert A \right\rvert ^{\frac{16}{7} }}  \right) ^{\frac{7}{9} } \left( \left\lvert A \right\rvert ^{3}  \right) ^{\frac{2}{9} }\\
& = \frac{\left\lvert A+A \right\rvert^{\frac{20}{9} } }{\left\lvert A \right\rvert ^{\frac{10}{9} }} 
\end{align*}

By Cauchy-Schwarz we have \(E(B) \geq \frac{\left\lvert B \right\rvert ^{4}}{\left\lvert B+B \right\rvert } \gg \frac{\left\lvert A \right\rvert ^{4}}{\left\lvert A+A \right\rvert }  \), and hence
\[
    \frac{\left\lvert A \right\rvert ^{4}}{\left\lvert A + A \right\rvert } \lesssim \frac{\left\lvert A+A \right\rvert^{\frac{20}{9} } }{\left\lvert A \right\rvert ^{\frac{10}{9} }}  \implies \left\lvert A+A \right\rvert \gtrsim \left\lvert A \right\rvert ^{\frac{46}{29} }
,\]
from which Theorem \ref{thm-CSUM} follows.

\bigskip

\section{Proof of Theorem \ref{thm-CDIFF}}\label{section:CDIFF}
We use the notation \(D  = A-A\), where \(A\) is clear.
We apply the argument of \cite{bloom} two times. The first time we follow his argument exactly, but to the energy \(E(A)\).
The improvement to the difference set bound is only due to Proposition \ref{prop-ea} below and Lemma \ref{lem:diff-proj}.

\begin{prop}\label{prop-ea}
For \(A\) convex,
\[
    E_{\frac{12}{5} } (A) \lesssim \left\lvert A \right\rvert ^{\frac{38}{15} } \left\lvert D  \right\rvert ^{\frac{4}{45} }
.\]
\end{prop}

We later obtain a bound for \(\left\lvert D \right\rvert \) in terms of \(E (A)\), in which case
the preceding proposition offers an advantage over interpolating with the bound \(E(A) \lesssim \left\lvert A \right\rvert ^{\frac{123}{50} }\),
which is the current best, and is due to Bloom in \cite{bloom}.

\begin{proof}[Proof of Proposition \ref{prop-ea}]
By a dyadic partitioning there is \(\xi \in \mathbb{R} \) such that, defining
\[
    X = \left\{ x : \delta_{A} (x) \in [\xi,2\xi) \right\} 
,\]
we have
\[
    E_{\frac{12}{5} } (A) \approx \xi^{\frac{12}{5} } \left\lvert X \right\rvert 
.\]

By the definition of \(X\),
\[
    \xi \left\lvert X \right\rvert \leq \sum _{x \in X} \delta_{A} (x) = \sum _{a \in A} \delta_{A,X} (a)
,\]
where the last equality follows from the fact that both sums count solutions to
\[
    a_1 - a_2 = x ,~ a_{i} \in A,~ x \in X
.\]

By Cauchy-Schwarz and the definition of \(\delta_{A,X}  \), we have
\[
    \frac{\xi^{2} \left\lvert X \right\rvert ^{2}}{\left\lvert A \right\rvert } \leq \sum _{a \in A} \delta_{A,X} (a)^{2} = \sum _{a \in A} \left[ \sum _{x \in X} 1_{A} (a + x)  \right] ^{2} = \sum _{x_1,x_2 \in X} \sum _{a \in A} 1_{A} (a + x_1) 1_{A} (a + x_2)
.\]

Note that, in the rightmost sum,
\[
    a + x_1 ,~ a + x_2 \in A \implies  x_1 - x_2 \in A-A
,\]
and so
\[
    \sum _{x_1,x_2 \in X} \sum _{a \in A} 1_{A} (a + x_1) 1_{A} (a + x_2)= \sum _{\substack{ x_1, x_2 \in X \\ x_1 - x_2 \in D }} \sum _{a \in A} 1_{A} (a + x_1)1_{A} (a + x_2)
.\]

Applying Cauchy-Schwarz again gives
\begin{equation} \label{eq:diff-big-CS}
    \frac{\xi^{4} \left\lvert X \right\rvert ^{4}}{\left\lvert A \right\rvert ^{2}} \leq \left\lvert \left\{ (x_1,x_2)\in X^{2} : x_1-x_2 \in D \right\}  \right\rvert \cdot \sum _{\substack{ x_1,x_2 \in X \\ x_1 - x_2 \in D }} \left[ \sum _{a \in A} 1_{A} (a + x_1) 1_{A} (a + x_2) \right] ^{2}.
\end{equation}  

By the definition of \(\delta_{X} \), we have
\[
    \left\lvert \left\{ (x_1,x_2)\in X^{2} : x_1-x_2\in D \right\}  \right\rvert = \sum _{d \in D} \delta_{X} (d)
.\]

Expanding and regrouping, we have
\begin{align*}
    & \sum _{\substack{ x_1,x_2 \in X \\ x_1 - x_2 \in D }} \left[ \sum _{a \in A} 1_{A} (a + x_1) 1_{A} (a + x_2) \right] ^{2} \\
    & \leq \sum _{x_1,x_2 \in X} \sum _{a_1,a_2 \in A} 1_{A} (a_1 + x_1)1_{A} (a_1 + x_2) 1_{A} (a_2 + x_1)1_{A} (a_2 + x_2)  \\
    & = \sum _{a_1,a_2 \in A} \left[ \sum _{x \in X} 1_{A} (a_1 + x)1_{A} (a_2 + x) \right] ^{2} \leq \sum _{a_1,a_2 \in A} \delta_{A} (a_1-a_2)^{2}.
\end{align*}
Partitioning the final sum as
\[
    \sum _{a_1,a_2 \in A} \delta_{A} (a_1-a_2)^{2} = \sum _{d \in D} \delta_{A} (d)^{2} \cdot \delta_{A} (d) = E_3(A)
,\]
we see that
\[
    \sum _{\substack{ x_1,x_2 \in X \\ x_1 - x_2 \in D }} \left[ \sum _{a \in A} 1_{A} (a + x_1) 1_{A} (a + x_2) \right] ^{2} \leq E_3(A) \lesssim \left\lvert A \right\rvert ^{3} 
,\]
where the last inequality follows from Lemma \ref{lem:convex-szt}.

Substituting into (\ref{eq:diff-big-CS}) gives
\[
    \frac{\xi^{4}\left\lvert X \right\rvert ^{4}}{\left\lvert A \right\rvert ^{2}} \lesssim \left\lvert A \right\rvert ^{3} \cdot \sum _{d \in D} \delta_{X} (d) = \left\lvert A \right\rvert ^{3} \cdot  \sum _{x \in X} \delta_{X,D} (x)
.\]

Seeing that, for \(x \in X,~ \delta_{A} (x) \geq \xi\), we have
\[
    \frac{\xi^{5}\left\lvert X \right\rvert ^{4}}{\left\lvert A \right\rvert ^{2}} \lesssim \left\lvert A \right\rvert ^{3}  \cdot \sum _{x \in X} \delta_{A} (x)\delta_{X,D} (x)
.\]

See that
\[
    \sum _{x \in X} \delta_{A} (x) \delta_{X,D} (x)   \leq \sum _{x} \delta_{A,D} (x) \delta_{A,X} (x)
,\]
as the LHS counts solutions to
\[
    a_1 - a_2 = x_1 = x_2 - d ,~ x_{i} \in X ,~ a_{i} \in A,~ d \in D
,\]
and
\[
    a_1 - a_2 = x_2 - d \iff a_2 - d = a_1 - x_2
.\]

By \holder's inequality,
\[
    \sum _{x} \delta_{A,D} (x) \delta_{A,X} (x)\leq E_3(A,D)^{\frac{1}{3} }E_{\frac{3}{2} } (A,X)^{\frac{2}{3} }
.\]

Using Lemma \ref{lem:convex-szt} gives
\[
    \frac{\xi^{5} \left\lvert X \right\rvert ^{4}}{\left\lvert A \right\rvert ^{2}} \leq \left\lvert A \right\rvert ^{3} \cdot E_3(A,D) ^{\frac{1}{3} }E_{\frac{3}{2} } (A,X)^{\frac{2}{3} } \lesssim \left\lvert A \right\rvert^{4} \left\lvert D \right\rvert ^{\frac{2}{3} }\left\lvert X \right\rvert ^{\frac{5}{6} }
,\]
or
\begin{equation} \label{eq:diff-to-interpolate}
\xi^{5} \left\lvert X \right\rvert ^{\frac{19}{6} }\lesssim \left\lvert A \right\rvert ^{6}\left\lvert D \right\rvert ^{\frac{2}{3} }
\end{equation}

It follows from Lemma \ref{lem:convex-szt} that
\[
    \xi^{3} \left\lvert X \right\rvert \leq \sum _{x \in X} \delta_{A} (x)^{3} \lesssim \left\lvert A \right\rvert ^{3} 
,\]
and interpolating this with (\ref{eq:diff-to-interpolate}), we have
\[
    \left( \xi^{\frac{12}{5} } \left\lvert X \right\rvert    \right)^{\frac{15}{2} }=  \left( \xi ^{3} \left\lvert X \right\rvert \right)^{\frac{13}{3} } \left( \xi^{5}\left\lvert X \right\rvert ^{\frac{19}{6} } \right)  \lesssim \left\lvert A \right\rvert ^{19}\left\lvert D \right\rvert ^{\frac{2}{3} }
.\]

Using \(\xi^{\frac{12}{5} }\left\lvert X \right\rvert \approx E_{\frac{12}{5} } (A)\) gives
\[
    E_{\frac{12}{5} } (A) \lesssim \left\lvert A \right\rvert ^{\frac{38}{15} } \left\lvert D \right\rvert ^{\frac{4}{45} }
.\]
\end{proof}

From Lemma \ref{lem:diff-proj}, it will suffice to find an upper bound on \(E(A,P)\).
Indeed, Lemma \ref{lem:diff-proj} gives
\begin{equation} \label{eq:from-lemma-diff-proj}
    \left\lvert A \right\rvert ^{6} \ll E_3(A) \cdot \sum _{x \in P} \delta_{P} (x),
\end{equation}
where
\[
    P = \left\{ x \in D : \delta_{A} (x) \geq \frac{1}{11} \cdot \frac{\left\lvert A \right\rvert ^{2}}{\left\lvert D \right\rvert }  \right\} 
.\]
Using Lemma \ref{lem:convex-szt}, \(E_3(A) \lesssim \left\lvert A \right\rvert ^{3} \), and using the definition of \(P\),
\[
    \frac{\left\lvert A \right\rvert ^{2}}{\left\lvert D \right\rvert } \cdot \sum _{x \in P} \delta_{P} (x) \leq \sum _{x \in P} \delta_{A} (x) \delta_{P} (x) \leq E(A,P)
.\]
Substituting both of these into (\ref{eq:from-lemma-diff-proj}), we see that
\begin{equation} \label{eq:diff-proj-energy}
    \frac{\left\lvert A \right\rvert ^{5}}{\left\lvert D \right\rvert } \lesssim E(A,P),
\end{equation}
so it suffices to find an upper bound on \(E(A,P)\).

We proceed now almost identically to \cite{bloom}, the only changes being the use of 
Proposition \ref{prop-ea} and a change in how 
H\"older inequality is used. 

We use
\[
    \# \left\{ a - t = p - d \right\} \leq E_{\frac{3}{2} } (A,T) ^{\frac{2}{3} }E_3(P,D)^{\frac{1}{3} }
\]
as opposed to
\[
    \# \left\{ a - t = p - d \right\} \leq E_{3 } (A,T)^{\frac{1}{3} } E_{\frac{3}{2} } (P,D)^{\frac{2}{3} }
,\]
which takes advantage of (\ref{eq:diff-proj-energy}) being a lower bound on \(E(A,P)\) as opposed to 
\[
    \# \left\{ a_1 - d = a_2 - p : a_1,a_2 \in A ,~ d \in D,~ p \in P \right\}
,\]
as it would be in \cite{bloom}.

We proceed in bounding \(E(A,P)\).
By a dyadic partitioning, there is \(\eta \in \mathbb{R} \) such that, defining
\[
    T = \left\{ x \in A  - P : \delta_{A,P} (x) \in [\eta,2\eta) \right\} 
,\]
we have
\[
    E(A,P)\approx \eta ^{2}\left\lvert T \right\rvert 
.\]
By the definition of \(T\),
\[
    \eta \left\lvert T  \right\rvert \leq \sum _{x \in T } \delta_{A,P} (x) = \sum _{x \in P} \delta_{A,T } (x)  
,\]
where the last equality follows from the fact that both sums count solutions to
\[
     a - p = t ,~ a \in A,~ p \in P ,~ t \in T
.\]

By Cauchy-Schwarz,
\[
    \frac{\eta^{2} \left\lvert T \right\rvert ^{2}}{\left\lvert P \right\rvert  }  \leq \sum _{x \in P} \delta_{A,T } (x)^{2} = \sum _{x \in P} \left[ \sum _{t \in T } 1_{A} (x + t) \right] ^{2} = \sum _{t_1,t_2 \in T} \sum _{x \in P} 1_{A} (x + t_1) 1_{A} (x + t_2)   
.\]

Note that, in the rightmost sum,
\[
    x + t_1 ,~ x + t_2 \in A \implies  t_1 - t_2 \in A-A
,\]
and so
\[
    \sum _{t_1,t_2 \in T} \sum _{x \in P} 1_{A} (x + t_1) 1_{A} (x + t_2)= \sum _{\substack{ t_1, t_2 \in T \\ t_1 - t_2 \in D }} \sum _{x \in P} 1_{A} (x + t_1)1_{A} (x + t_2)
.\]

Hence, by Cauchy-Schwarz,
\begin{equation} \label{eq:large-2}
    \frac{\eta ^{4}\left\lvert T \right\rvert ^{4}}{\left\lvert P \right\rvert ^{2}} \leq \left[ \sum _{x \in D} \delta_{T } (x) \right] \left[ \sum _{t_1,t_2 \in T}  \left( \sum _{x \in P} 1_{A} (x + t_1)1_{A} (x + t_2) \right) ^{2} \right] .
\end{equation}

Expanding and regrouping,
\begin{align*}
& \sum _{t_1,t_2 \in T } \left( \sum _{x \in P} 1_{A} (x + t_1) 1_{A} (x+t_2) \right) ^{2}\\
& = \sum _{t_1,t_2 \in T } \sum _{x_1,x_2 \in P} 1_{A} (x_1 + t_1)1_{A} (x_1 + t_2) 1_{A} (x_2 + t_1) 1_{A} (x_2 + t_2) \\
& = \sum _{x_1,x_2 \in P} \left( \sum _{t \in T } 1_{A} (x_1 + t) 1_{A} (x_2 + t) \right) ^{2} \leq \sum _{x_1,x_2 \in P} \delta_{A} (x_1-x_2)^{2}.
\end{align*}
We can partition the last sum as
\[
    \sum _{x_1,x_2 \in P} \delta_{A} (x_1 - x_2)^{2} = \sum _{x \in P - P} \delta_{A} (x)^{2} \delta_{P}(x)  
.\]

With this, (\ref{eq:large-2}) becomes
\begin{equation} \label{eq:double-bound}
    \frac{\eta^{4} \left\lvert T  \right\rvert ^{4}}{\left\lvert P \right\rvert ^{2}} \lesssim \left( \sum _{x \in D} \delta_{T} (x)  \right)  \left(  \sum _{x} \delta_{A} (x) ^{2}\delta_{P} (x)  \right) 
\end{equation}

Bounding the first term of (\ref{eq:double-bound}), we have
\[
    \sum _{x \in D} \delta_{T} (x) = \sum _{x \in T} \delta_{T,D} (x)
.\]
By the definition of \(T\), we have
\[
    \eta \cdot \sum _{x \in T} \delta_{T,D} (x) \leq \sum _{x \in T} \delta_{T,D} (x) \delta_{A,P} (x) \leq \sum _{x}  \delta_{T,D} (x) \delta_{A,P} (x)
.\]
Using \holder's inequality on the last term above and substituting gives
\begin{equation} \label{eq:first-term}
    \sum _{x \in D} \delta_{T} (x) \leq \frac{1}{\eta} E_{\frac{3}{2} } (A,T)^{\frac{2}{3} } E_{3} (P,D)^{\frac{1}{3} }.
\end{equation}

We proceed with bounding \(E_{3} (P,D)^{\frac{1}{3} }\). Firstly, let \(B_{i} = \left\{ x \in A + D : \sigma_{A,D} (x) \in [2^{i},2^{i+1}) \right\} \),
and let \(\tau_{x} :\mathbb{R}  \to \mathbb{R} \) be the ``translation'' function, so
\[
    \tau_{x} (y) = y + x
.\]
Seeing that, for any \(x \in P-D\),
\[
    \delta_{P,D} (x) \ll \frac{\left\lvert D \right\rvert }{\left\lvert A \right\rvert ^{2}} \cdot r_{A-A - D} (x) 
\]
we have
\[
   E_{3} (P,D)^{\frac{1}{3} } = \left\lVert \delta_{P,D}  \right\rVert _{3}  \ll \frac{\left\lvert D \right\rvert }{\left\lvert A \right\rvert ^{2}} \left\lVert r_{A-A-D} \right\rVert _{3}
.\]

Using the definitions of \(\tau\) and \(r_{A-A-D} \), we rewrite the RHS as
\[
    \frac{\left\lvert D \right\rvert }{\left\lvert A \right\rvert ^{2}}  \left\lVert r_{A-A-D}  \right\rVert _{3}  = \frac{\left\lvert D \right\rvert }{\left\lvert A \right\rvert ^{2}}  \left\lVert \sum _{\substack{ a \in A \\ d \in D }}  1_{A} \circ \tau_{a + d}   \right\rVert _{3} 
.\]
Partitioning the sum over \(A+D\) and \(B_{i} \) respectively gives
\[
    \left\lVert \sum _{\substack{ a \in A \\ d \in D }}  1_{A} \circ \tau_{a + d}   \right\rVert _{3} = \left\lVert \sum _{x \in A + D} 1_{A} \circ \tau_{x} \cdot \sigma_{A,D}  (x) \right\rVert _{3} = \left\lVert \sum _{i \in [\log _{2} \left\lvert A \right\rvert ] } \sum _{x \in B_{i} } 1_{A} \circ \tau_{x} \cdot \sigma_{A,D} (x) \right\rVert _{3}
.\]
By Minkowski's inequality and definition of \(B_{i} \) respectively,
\begin{align*}
    \left\lVert \sum _{i \in [\log _{2} \left\lvert A \right\rvert ] } \sum _{x \in B_{i} } 1_{A} \circ \tau_{x} \cdot \sigma_{A,D} (x) \right\rVert _{3}&  \leq \sum _{i \in [\log _{2} \left\lvert A \right\rvert ]} \left\lVert \sum _{x \in B_{i} } 1_{A} \circ \tau_{x} \cdot \sigma_{A,D} (x) \right\rVert _{3} \\
    & \leq \sum _{i \in [\log _{2} \left\lvert A \right\rvert ]} 2^{i+1} \left\lVert \sum _{x \in B_{i} } 1_{A} \circ \tau_{x}  \right\rVert _{3}.
\end{align*}
By the definition of \(\delta_{A,B_{i} } \),
\[
    \left[ \sum _{x \in B_{i} } 1_{A} \circ \tau_{x} \right] (y) = \delta_{A,B_{i} } (y)
,\]
and hence
\[
    \sum _{i \in [\log _{2} \left\lvert A \right\rvert ]} 2^{i+1} \left\lVert \sum _{x \in B_{i} } 1_{A} \circ \tau_{x}  \right\rVert _{3}  = \sum _{i \in [ \log _{2} \left\lvert A \right\rvert ]} 2^{i+1} \left\lVert \delta_{A,B_{i} }  \right\rVert _{3} 
.\]
Using Lemma \ref{lem:convex-szt} gives
\[
    \sum _{i \in [\log _{2} \left\lvert A \right\rvert ]} 2^{i+1} \left\lVert \delta_{A,B_{i} }  \right\rVert _{3} \lesssim \left\lvert A \right\rvert ^{\frac{1}{3} } \sum _{i \in [\log _{2} \left\lvert A \right\rvert ]} 2^{i+1}  \left\lvert B_{i}  \right\rvert ^{\frac{2}{3} }
.\]
See that there is \(j \in \mathbb{N} \) such that
\[
    \sum _{i \leq \log _{2} \left\lvert A \right\rvert } 2^{i+1} \left\lvert B_{i}  \right\rvert ^{\frac{2}{3} } \leq \log _{2} \left\lvert A \right\rvert 2^{j+1} \left\lvert B_{j}  \right\rvert ^{\frac{2}{3} } \lesssim \left( \sum _{x \in B_{j} } \sigma_{A,D} (x)^{\frac{3}{2} } \right) ^{\frac{2}{3} }
.\]
Using Lemma \ref{lem:convex-szt},
\[
    \left( \sum _{x \in B_{j} } \sigma_{A,D}    (x)^{\frac{3}{2} } \right)^{\frac{2}{3} } \leq  E_{\frac{3}{2} } (A,\left( -D \right) )^{\frac{2}{3} } \lesssim \left\lvert A \right\rvert ^{\frac{2}{3} } \left\lvert D \right\rvert ^{\frac{5}{6} }
.\]

Finally, combining all the work above yields
\[
    E_{3} (P,D)^{\frac{1}{3} } \lesssim  \frac{\left\lvert D \right\rvert }{\left\lvert A \right\rvert ^{2}} \cdot \left\lvert A \right\rvert ^{\frac{1}{3} } \cdot \left\lvert A \right\rvert ^{\frac{2}{3} } \left\lvert D \right\rvert ^{\frac{5}{6} } = \frac{\left\lvert D \right\rvert ^{\frac{11}{6} }}{\left\lvert A \right\rvert }
.\]

Returning to (\ref{eq:first-term}), using Lemma \ref{lem:convex-szt} for \(E_{\frac{3}{2} } (A,T)^{\frac{2}{3} }\), and the inequality above, we obtain
\begin{equation} \label{eq:bloom-arg-part-1}
    \sum _{x \in D} \delta_{T} (x) \leq \frac{1}{\eta} E_{\frac{3}{2} } (A,T)^{\frac{2}{3} } E_{3} (P,D)^{\frac{1}{3} } \lesssim \frac{1}{\eta} \cdot  \left\lvert A \right\rvert ^{\frac{2}{3} }\left\lvert T \right\rvert ^{\frac{5}{6} } \cdot \frac{\left\lvert D \right\rvert ^{\frac{11}{6} }}{\left\lvert A \right\rvert }  = \frac{\left\lvert T \right\rvert ^{\frac{5}{6} }}{\eta}    \cdot \frac{\left\lvert D \right\rvert ^{\frac{11}{6} }}{\left\lvert A \right\rvert ^{\frac{1}{3} }} .
\end{equation}

We proceed with the second term of (\ref{eq:double-bound}), namely
\[
    \sum _{x} \delta_{A} (x)^{2}\delta_{P} (x)
.\]
By a dyadic pigeonholing we see that \( \exists \upsilon \in \mathbb{R}     \) such that, defining
\[
    U = \left\{ x : \delta_{A} (x) \in [\upsilon,2\upsilon) \right\}
,\]
we have
\[
    \sum _{x} \delta_{A} (x) ^{2} \delta_{P} (x) \approx \sum _{x \in U} \delta_{A}(x)^{2} \delta_{P} (x) \asymp \upsilon^{2} \sum _{x \in U} \delta_{P} (x) 
.\]

See that
\[
    \sum _{x \in U} \delta_{P} (x) = \sum _{x \in P} \delta_{P,U} (x)
,\]
as they both count solutions to
\[
    p_1-p_2 = u ,~ p_{i} \in P,~ u \in U
.\]

By the definition of \(P\),
\[
    \frac{\left\lvert A \right\rvert ^{2}}{\left\lvert D \right\rvert } \sum _{x \in P} \delta_{P,U} (x) \ll \sum _{x \in P} \delta_{A} (x) \delta_{P,U} (x)
.\]
See that
\[
    \sum _{x}  \delta_{A} (x) \delta_{P,U} (x) = \sum _{x} \delta_{A,P} (x)\delta_{A,U} (x)
,\]
as they both count solutions to
\[
    a_1 - a_2 = p - u ,~ a_{i} \in A,~ p \in P, u \in U
.\]

\holder's inequality gives
\[
    \sum _{x} \delta_{A,P} (x) \delta_{A,U} (x) \leq E_3(A,P)^{\frac{1}{3} } E_{\frac{3}{2} } (A,U)^{\frac{2}{3} }
,\]
and Lemma \ref{lem:convex-szt} gives
\[
    E_3(A,P)^{\frac{1}{3} } E_{\frac{3}{2} } (A,U)^{\frac{2}{3} } \lesssim \left\lvert A \right\rvert \left\lvert P \right\rvert ^{\frac{2}{3} }\left\lvert U \right\rvert ^{\frac{5}{6} }\leq \left\lvert A \right\rvert \left\lvert D \right\rvert ^{\frac{2}{3} }\left\lvert U \right\rvert ^{\frac{5}{6} }
.\]  

Combining these results gives
\[
    \sum _{x} \delta_{A} (x) ^{2}\delta_{P} (x) \lesssim \upsilon ^{2}\cdot \frac{\left\lvert D \right\rvert }{\left\lvert A \right\rvert ^{2}} \cdot \left\lvert A \right\rvert \left\lvert D \right\rvert ^{\frac{2}{3} }\left\lvert U \right\rvert ^{\frac{5}{6} } = \frac{\left\lvert D \right\rvert ^{\frac{5}{3} }}{\left\lvert A \right\rvert } \cdot \left( \upsilon^{\frac{12}{5} } \left\lvert U \right\rvert  \right) ^{\frac{5}{6} }
.\]
By the definition of \(\upsilon,U\),
\[
    \upsilon^{\frac{12}{5} } \left\lvert U \right\rvert \asymp \sum _{x \in U} \delta_{A} (x)^{\frac{12}{5} } \leq \sum _{x}  \delta_{A} (x) ^{\frac{12}{5} } = E_{\frac{12}{5} } (A)
,\]
so substituting gives
\begin{equation} \label{eq:bloom-arg-part-2}
    \sum _{x} \delta_{A} (x) ^{2}\delta_{P} (x) \lesssim \frac{\left\lvert D \right\rvert ^{\frac{5}{3} }}{\left\lvert A \right\rvert } \cdot E_{\frac{12}{5} } (A)^{\frac{5}{6} }
\end{equation}

Combining (\ref{eq:bloom-arg-part-1}) and (\ref{eq:bloom-arg-part-2}), (\ref{eq:double-bound}) becomes
\[
    \frac{\eta^{4} \left\lvert T  \right\rvert ^{4}}{\left\lvert P \right\rvert ^{2}} \lesssim \left( \frac{\left\lvert T \right\rvert ^{\frac{5}{6} }}{\eta}     \cdot \frac{\left\lvert D \right\rvert ^{\frac{11}{6} }}{\left\lvert A \right\rvert ^{\frac{1}{3} }}  \right)  \left( \frac{\left\lvert D \right\rvert ^{\frac{5}{3} }}{\left\lvert A \right\rvert } E_{\frac{12}{5} } (A)^{\frac{5}{6} } \right) 
,\]
or
\begin{equation} \label{eq:to-interpolate}
    \eta^{5} \left\lvert T \right\rvert ^{\frac{19}{6} } \lesssim \frac{\left\lvert D \right\rvert ^{\frac{11}{2} }}{\left\lvert A \right\rvert ^{\frac{4}{3} }} E_{\frac{12}{5} } (A) ^{\frac{5}{6} } .
\end{equation}

See that, using Lemma \ref{lem:convex-szt},
\[
    \eta ^{3} \left\lvert T \right\rvert \asymp \sum _{x \in T} \delta_{A,P} (x)^{3}  \leq E_3(A,P)\lesssim \left\lvert A \right\rvert \left\lvert P \right\rvert ^{2} \leq \left\lvert A \right\rvert \left\lvert D \right\rvert ^{2}  
,\]
so in particular
\[
    \eta^{3}  \left\lvert T \right\rvert \lesssim  \left\lvert A \right\rvert \left\lvert D \right\rvert ^{2}
.\]

Interpolating with (\ref{eq:to-interpolate}) gives
\[
    \left( \eta ^{3} \left\lvert T \right\rvert      \right) ^{\frac{4}{3} } \eta^{5} \left\lvert T \right\rvert ^{\frac{19}{6} } \lesssim \left( \left\lvert A \right\rvert \left\lvert D \right\rvert ^{2} \right) ^{\frac{4}{3} } \left(   \frac{\left\lvert D \right\rvert ^{\frac{11}{2} }}{\left\lvert A \right\rvert ^{\frac{4}{3} }} E_{\frac{12}{5} } (A) ^{\frac{5}{6} }\right) 
,\]
or by simplifying,
\[
    \eta^{9}\left\lvert T \right\rvert ^{\frac{9}{2} }\lesssim \left\lvert D \right\rvert ^{\frac{49}{6} } E_{\frac{12}{5} } (A)^{\frac{5}{6} }
.\]
By the definition of \(\eta, T\),
\[
    \eta^{9} \left\lvert T \right\rvert ^{\frac{9}{2} } = \left( \eta ^{2}\left\lvert T \right\rvert  \right) ^{\frac{9}{2} } \approx E(A,P)^{\frac{9}{2} } 
\]
and substituting gives
\[
    E(A,P)\lesssim \left\lvert D \right\rvert ^{\frac{49}{27} }  E_{\frac{12}{5} } (A)^{\frac{5}{27} } 
.\]

Applying Proposition \ref{prop-ea} and substituting into (\ref{eq:diff-proj-energy}) gives
\[
    \frac{\left\lvert A \right\rvert ^{5}}{\left\lvert D \right\rvert } \lesssim \left\lvert D \right\rvert ^{\frac{49}{27} } \left( \left\lvert A \right\rvert ^{\frac{38}{15} } \left\lvert D \right\rvert ^{\frac{4}{45} } \right) ^{\frac{5}{27} }
,\]
or
\[
    \left\lvert D \right\rvert \gtrsim \left\lvert A \right\rvert ^{\frac{1101}{688} } = \left\lvert A \right\rvert ^{\frac{8}{5}  + \frac{1}{3440}   }
,\]
from which Theorem \ref{thm-CDIFF} follows.

\bigskip

\bibliographystyle{alpha}
\bibliography{bibli}

\end{document}